\documentclass[reqno,psamsfonts,11pt]{gen-j-l}
\usepackage{amssymb,amsmath,amsbsy,amsthm,esint,setspace,graphicx,color}
\usepackage{hyperref}
\usepackage{amsrefs}
\usepackage{enumitem}

\def\R{{\mathbb R}}

\newtheorem{theorem}{Theorem}[section]
\newtheorem{lemma}[theorem]{Lemma}
\newtheorem{proposition}[theorem]{Proposition}
\newtheorem{corollary}[theorem]{Corollary}

\newtheorem{example}{Example}[section]

\newtheorem*{remark}{Remark}

\begin{document}

\title[Boltzmann's equation and the Wigner transform]
{Local well-posedness for Boltzmann's equation and the 
Boltzmann hierarchy via Wigner transform}

\author{Thomas Chen, Ryan Denlinger, and Nata{\v s}a Pavlovi{\' c}}

\begin{abstract}
We use the dispersive properties of the linear Schr{\" o}dinger equation to
 prove local well-posedness results for the Boltzmann equation and
the related Boltzmann hierarchy, set in the spatial
domain $\mathbb{R}^d$ for $d\geq 2$. The proofs are based on the use
of the (inverse) Wigner transform along with the spacetime Fourier transform.
The norms for the initial data $f_0$ are weighted versions
of the Sobolev spaces $L^2_v H^\alpha_x$ with $\alpha \in
\left( \frac{d-1}{2},\infty\right)$.
Our main results are local well-posedness for the Boltzmann equation for
cutoff Maxwell molecules and hard spheres, as well as local well-posedness
for the Boltzmann hierarchy for cutoff Maxwell molecules (but not hard
spheres); the latter result holds without any factorization assumption for
the initial data.
\end{abstract}

\maketitle

\section{Introduction}
\label{intro}

Boltzmann's equation is an evolutionary partial differential equation
(PDE) which describes the behavior of a dilute gas of identical particles
in a specific scaling limit. The equation describes the
time evolution of a density function $f(t,x,v)\geq 0$, where
$x,v \in \mathbb{R}^d$ are the position and velocity of a typical particle.

The Cauchy problem for Boltzmann's equation is one of the fundamental
mathematical problems in kinetic theory and it may be written in the following
form:
\begin{align}
\label{eq:boltz}
& \left( \partial_t + v\cdot \nabla_x\right) f(t,x,v) =
Q (f,f) (t,x,v) \\
& f(0,x,v) = f_0(x,v)
\end{align}
where the \emph{collision operator} $Q$ is defined
as follows:
\begin{equation}
\label{eq:coll-1}
\begin{aligned}
& Q (f,f) (t,x,v) = \int_{\mathbb{R}^d} \int_{\mathbb{S}^{d-1}}
d\omega dv_2
\mathbf{b} 
\left( |v-v_2|,\omega\cdot\frac{v-v_2}{|v-v_2|}\right)\times \\
&\qquad \qquad \qquad \qquad \times \left(
f(t,x,v^*) f(t,x,v_2^*)-f(t,x,v)f(t,x,v_2)\right)
\end{aligned}
\end{equation}
The \emph{collision kernel} $\mathbf{b}$ is a function which depends on the
physical interaction between particles; pre-collisional and post-collisional
velocities are related by the following involutive transformation,
for $v,v_2 \in \mathbb{R}^d$ and fixed $\omega \in \mathbb{S}^{d-1}$:
\begin{equation*}
\begin{aligned}
v^* & = v + \left( \omega \cdot \left( v_2-v\right)\right)\omega\\
v_2^* & = v_2 - \left(\omega\cdot\left(v_2-v\right)\right)\omega
\end{aligned}
\end{equation*}

The most general known solutions of 
Boltzmann's equation are the  \emph{renormalized solutions} of \cite{DPL1989},
which exist 
globally in time for arbitrary data $f_0$ having finite mass,
second moments and entropy:
\begin{equation}\label{intro-ma-en-en}
\int_{\R^d \times \R^d} f_0(x,v) \, \left( 1 + |v|^2 + |x|^2 +
\log(1+f_0(x,v))\right) \, dx dv < +\infty.
\end{equation}
However, renormalized solutions have
many limitations; for instance, they are not known to solve the
Boltzmann equation in the usual distributional sense (which makes them
difficult to manipulate), nor are they known to be unique, nor are they
known to conserve energy. A different and very fruitful line of
investigation considers solutions close to an equilibrium distribution
of fixed temperature, see e.g.
 \cite{Uk1974,GS2011,AMUXY2011,Du2008,Gu2003,UY2006}.
 These solutions exist globally in
time and enjoy uniqueness and continuous dependence in appropriate
functional spaces; however, the theory only applies in a small neighborhood
of equilibrium. 

Henceforth we will not concern ourselves with the
(very difficult) problem of global well-posedness for Boltzmann's
equation. Instead we will be interested in the 
\emph{local} theory of well-posedness. Generally this means we want
to prove existence, uniqueness and continuous dependence of solutions,
locally in time and for large data,
with regularity as low as possible. See \cite{KS1984,Ar2011,AMUXY-Loc}
 for some existing theories of
local solutions for Boltzmann's equation. We especially refer to
Remark 1 of of \cite{AMUXY-Loc}, which provides (in the case of Grad
cut-off) a large data local well-posedness result which parallels
our Theorem \ref{thm:BE-LWP} when $\alpha > \frac{d}{2}$ in 
$d=3$.\footnote{We
are able to prove a \emph{conditional} local well-posedness result
when $f_0$ is in a weighted version of
$L^2_v H^\alpha_x$ 
with $\alpha > \frac{d-1}{2}$ (here conditional means that
uniqueness only holds 
assuming some auxiliary estimate satisfied by the constructed
solution). It is conceivable that the uniqueness is unconditional
when $\alpha > \frac{d}{2}$, 
cf. \cite{AMUXY-Loc}, but we do not pursue this issue.}
Our main intention, however, is not to investigate optimal regularity
spaces for solving Boltzmann's equation.
Rather, we intend to
demonstrate the close connection between Boltzmann's equation
and nonlinear Schr{\" o}dinger equations (NLS) in the density matrix
formulation; this connection has been
recognized implicitly for some time, but we wish to make it
quite explicit and to the best of our knowledge this is the first time 
such an explicit connection has been established.\footnote{We emphasize
that we \emph{do not} make use of any semiclassical limit.}
 The local well-posedness theory for NLS is by now
very mature and it is our hope that some tools which have been useful
for NLS will turn out
to be applicable to the corresponding problem for Boltzmann's equation.
If the theory can be made precise enough, it may turn out to be useful
for such problems as global well-posedness or the derivation of
Boltzmann's equation from deterministic particle systems.

Besides providing a new approach to proving local well-posedness for
Boltzmann's equation, we will also prove new results concerning the
\emph{Boltzmann hierarchy} for at least some collision kernels. 
The Boltzmann hierarchy is an infinite hierarchy of coupled PDE which
describes a gas of infinitely many particles, possibly accounting for
correlations between particles. For some class of collisional kernels, 
the Boltzmann hierarchy appears in the derivation of Boltzmann's equation
from classical system of many particles. See e.g. \cite{L1975,GSRT2014}. 
The connection between the Boltzmann
hierarchy and Boltzmann equation lies in the fact that the Boltzmann hierarchy admits a class of
factorized solutions with each factor being a solution to the Boltzmann equation.

The classical local well-posedness
result for the Boltzmann hierarchy is due to Lanford \cite{L1975}, who
assumes $L^\infty$ bounds on the initial data. Our results establish local
well-posedness in a functional setting much different than Lanford's;
in particular, we can work with
spaces that do not embed locally into $L^\infty$ in any variable.
Unfortunately we cannot report any new results concerning the Boltzmann
hierarchy for hard spheres; this is the topic of ongoing research.

The idea at the heart of our proofs is to take the inverse Wigner 
transform of Boltzmann's equation (resp. the Boltzmann hierarchy).
The transport operator 
$$\left(\partial_t + v\cdot\nabla_x\right)$$
is transformed into the linear Schr{\" o}dinger operator
$$\left( i \partial_t + \frac{1}{2} \left( \Delta_x - \Delta_{x^\prime} \right)\right),$$ 
and the nonlinear operator $Q(f,f)$ becomes a new
operator $B(\gamma,\gamma)$.
This puts us in
a situation where we can prove a bilinear estimate of the similar type
as the one proved in the work of Klainerman and
Machedon \cite{KM2008}. Subsequently we can employ
the iteration method inspired by the one of Chen and Pavlovi{\' c} \cite{PC2010};
these methods were originally devised for proving the local well-posedness of
the Gross-Pitaevskii hierarchy. In this paper we implement them at the level of the 
transformed Boltzmann equation as well as at the level of the transformed Boltzmann hierarchy.\footnote{in which case we also use the boardgame combinatorial 
argument as presented by Klainerman and Machedon \cite{KM2008}, which is a
reformulation of the combinatorial methods of Erd{\" o}s, Schlein and Yau,
\cite{ES2001,ESY2006,ESY2007}.} The main point that we make here is that the transformed Boltzmann equation 
becomes a nonlinear Schr\"{o}dinger equation, and the transformed Boltzmann hierarchy becomes a Schr\"{o}dinger 
type hierarchy (usually called Gross-Pitaevskii hierarchy) with nonlinearities that encode information about the interaction between particles 
encoded in the Boltzmann collision kernels. Once we are at the level of such nonlinear   Schr\"{o}dinger equation/hierarchy, 
we develop tools and emloy techniques for local well-posedness inspired by tools and techniques that have been recently introduced in the context of the Gross-Pitaevskii hierarchy. 

\subsection*{Organization of the paper}
Section \ref{sec:2} describes in detail the main results we will
prove, using the Wigner transform. 
Section \ref{sec:key} gives the
proof of a crucial proposition which is used to prove all our results,
and constitutes the main technical contribution
of the paper.  
Section \ref{sec:3} is devoted to
the proof of local well-posedness for the Boltzmann equation; this
result extends to cutoff Maxwell molecules, hard spheres, and
variable hard sphere models. Section \ref{sec:4} gives a brief outline
of the proof of local well-posedness for the Boltzmann hierarchy, including
the case of cutoff Maxwell molecules (but not hard spheres).

\subsection*{Acknowledgements}
\label{sec:acknowledgements}

The work of T.C. is supported by NSF grants DMS-1151414 (CAREER)
and DMS-1262411.
R. D. gratefully acknowledges support from a postdoctoral fellowship
at the University of Texas at Austin.
The work of N.P. is supported by NSF grant DMS-1516228.

\section{Main Results}
\label{sec:2}

\subsection{Notation and preliminaries} 

Given a function $f(x,v) \in L^2_{x,v}$ we define its inverse
Wigner transform $\gamma (x,x^\prime) \in L^2_{x,x^\prime}$ by the
following formula:
\begin{equation}
\gamma (x,x^\prime) = \int_{\mathbb{R}^d} f \left(\frac{x+x^\prime}{2},
v\right) e^{iv\cdot (x-x^\prime)} dv
\end{equation}
The inverse of the inverse Wigner transform is the usual Wigner
transform, namely:
\begin{equation}
f(x,v) = \frac{1}{(2\pi)^d} \int_{\mathbb{R}^d}
\gamma \left( x+\frac{y}{2},x-\frac{y}{2}\right)
e^{-iv\cdot y} dy
\end{equation}
All of our main results will be stated in terms of $\gamma$; in particular,
if we say $f(t)$ satisfies Boltzmann's equation, we mean that $\gamma(t)$
solves the Duhamel formula associated with the inverse Wigner transform
of the Boltzmann equation.

\begin{remark}
Note that if $\gamma (x,x^\prime) = \overline{\gamma (x^\prime,x)}$ for
all $x,x^\prime \in \mathbb{R}^d$,
then $f$ is everywhere real-valued; the converse also holds. In particular,
it is easy to check on the inverse Wigner side that $f$ is real-valued.
It is much less simple to determine whether $f$ is non-negative, and this
is an issue we do not address in the present work.
\end{remark}

Throughout the paper, we will assume that
$0 \leq \mathbf{b} \in L^\infty_A$ for some $A\in [0,1]$, where we have
defined
\begin{equation}
\left\Vert \mathbf{b} \right\Vert_{L^\infty_A} = 
\sup_{u \in \mathbb{R}^d,\; \omega\in\mathbb{S}^{d-1}}
\frac{\left|\mathbf{b}\left(|u|,\omega\cdot\frac{u}{|u|}\right)\right|}
{1+|u|^A}
\end{equation}
We will require the Fourier transform of the collision kernel, which
is written
\begin{equation}
\hat{\mathbf{b}}^\omega (\xi) = \int_{\mathbb{R}^d}
\mathbf{b} \left(|u|,\omega\cdot\frac{u}{|u|}\right)
e^{-iu\cdot \xi} du
\end{equation}
Note that $\hat{\mathbf{b}}^\omega$ is a tempered distribution in general.
Special cases include  $\mathbf{b}\equiv 1$ with $A=0$ (Maxwell molecules
with angular cut-off), $\mathbf{b}=\left[\omega\cdot u \right]_+$ with
$A=1$ (hard spheres),
and $0<A<1$ for variable hard sphere models. Not all results will apply
for the full range $A\in [0,1]$.

We introduce the weighted Sobolev spaces which define our functional
setting. Let $\hat{\gamma}$ denote the Fourier transform of
$\gamma$:
\begin{equation}
\hat{\gamma} (\xi,\xi^\prime) = \int_{\mathbb{R}^d\times\mathbb{R}^d}
e^{-ix\cdot \xi} e^{-i x^\prime \cdot \xi^\prime}
\gamma(x,x^\prime) dx dx^\prime
\end{equation}
Then, for any $\alpha,\beta,\kappa \geq 0$, and any $\sigma > 0$,
\begin{equation}
\left\Vert \gamma (x,x^\prime) \right\Vert_{H^{\alpha,\beta,\sigma,\kappa}}
= \left\Vert \left< \xi+\xi^\prime\right>^\alpha
\left<\xi-\xi^\prime\right>^\beta
e^{\kappa \left< \xi-\xi^\prime\right>^{\frac{1}{\sigma}}}
\hat{\gamma} (\xi,\xi^\prime)\right\Vert_{L^2_{\xi,\xi^\prime}}
\end{equation}
Note that this norm is equivalent to the following norm for the
classical densities:
\begin{equation}
\begin{aligned}
\left\Vert \left< 2v\right>^\beta 
e^{\kappa \left<2v\right>^{\frac{1}{\sigma}}}
\left( 1-\Delta_x\right)^{\frac{\alpha}{2}} f(x,v)
\right\Vert_{L^2_{x,v}}
\end{aligned}
\end{equation}

\begin{remark}
We emphasize that we can allow $\kappa = 0$ for some of our
results, e.g. the case of cutoff Maxwell molecules. We always
require $\kappa > 0$ in the case of hard spheres.
\end{remark}

\subsection{Warm-up: Free transport.}

We present a few brief remarks on the free transport equation before
turning to our main results. The  main point we wish to make is that
if $f(t,x,v)$ solves the equation
\begin{equation}
\label{eq:FTR}
\left( \partial_t + v \cdot \nabla_x\right) f = 0
\end{equation}
then the inverse Wigner transform $\gamma (t,x,x^\prime)$ 
satisfies the following linear Schr{\" o}dinger equation:
\begin{equation}
\label{eq:LSCH}
\left( i \partial_t + \frac{1}{2} \left(
\Delta_x - \Delta_{x^\prime} \right) \right) \gamma (t,x,x^\prime) = 0
\end{equation}
We emphasize that this correspondence does not rely on any
semiclassical limit.

\begin{example}
If $\gamma (t,x,x^\prime) =
e^{i k \cdot (x-x^\prime)}$ for some $k \in \mathbb{R}^d$,
then $\gamma$ solves (\ref{eq:LSCH}) and
$f (t,x,v) = c \delta(v-k)$ solves (\ref{eq:FTR}).
\end{example}

\begin{example}
If $\gamma(t,x,x^\prime) = |t|^{-d}
e^{i(|x|^2 - |x^\prime|^2)/(2t)}$, then
$\gamma$ solves (\ref{eq:LSCH}) and $f(t,x,v) = c \delta(x-vt)$
solves (\ref{eq:FTR}).
\end{example}

\begin{example}
If $f(t,x,v) = \delta(x-vt) \delta(v-v_0)$, for a fixed
$v_0 \in \mathbb{R}^d$, then $f$ solves (\ref{eq:FTR}); moreover,
the classical state (position and velocity) is known exactly.
In any case,
$\gamma$ exists as a distribution; for any $u(t,x,x^\prime) \in 
\mathcal{C}_0^\infty (\mathbb{R} \times \mathbb{R}^d \times
\mathbb{R}^d)$ we have
\begin{equation}
\begin{aligned}
& \int_{\mathbb{R} \times \mathbb{R}^d \times \mathbb{R}^d}
\gamma (t,x,x^\prime) u (t,x,x^\prime) dt dx dx^\prime = \\
&\qquad \qquad \qquad \qquad \qquad 
 = c \int_{\mathbb{R} \times \mathbb{R}^d}
e^{2 i v_0 \cdot z} u(t, v_0 t + z, v_0 t - z) dz dt
\end{aligned}
\end{equation}
Equivalently, $\gamma (t,x,x^\prime) =
\delta \left( \frac{x+x^\prime}{2} - v_0 t\right)
e^{i v_0 \cdot (x-x^\prime)}$. If $v_0 = 0$ then $\gamma$ obviously
solves (\ref{eq:LSCH}); by a Galilean shift, $\gamma$ solves
(\ref{eq:LSCH}) for arbitrary $v_0 \in \mathbb{R}^d$. Therefore, the
``fundamental solution'' for (\ref{eq:FTR}) transforms into 
a solution of (\ref{eq:LSCH}) under the inverse Wigner transform.  
We conclude that \emph{any} classical state (evolving under free transport)
 can be represented by
a distribution $\gamma (t,x,x^\prime)$ (evolving via a linear 
Schr{\" o}dinger equation).
Let us also point out that the inverse Wigner transform regarded as
a map
$f \in L^2_{x,v} \mapsto \gamma \in L^2_{x,x^\prime}$ is an
isometric isomorphism; and, these spaces are preserved by either
(\ref{eq:FTR}) or (\ref{eq:LSCH}) respectively. Hence the equivalence of
(\ref{eq:FTR}) and (\ref{eq:LSCH}) is reflected at the $L^2$ level
of regularity.
\end{example}

\begin{remark}
If 
$\phi (t,x)$ solves the Schr{\" o}dinger equation
\begin{equation}
\left( i \partial_t + \frac{1}{2} \Delta_x \right) \phi (t,x)
= 0
\end{equation} 
then the function $\gamma (t,x,x^\prime) = 
\phi (t,x) \overline{\phi(t,x^\prime)}$ solves
(\ref{eq:LSCH}) and the Wigner transform  $f$ solves
(\ref{eq:FTR}) (though $f$ in this case \emph{need not} be 
non-negative).
\end{remark}

We now prove the equivalence of (\ref{eq:FTR}) and 
(\ref{eq:LSCH}) at the $L^2$ level of regularity. (The same
result holds if $f,\gamma$ are tempered distributions, and
the proof is the same.)
\begin{lemma}
Suppose 
\begin{equation}
f \in L^1 \left( [0,T], L^2 (\mathbb{R}^d_x \times
\mathbb{R}^d_v) \right)
\end{equation}
and let 
\begin{equation}
\gamma \in L^1 \left( [0,T],L^2
(\mathbb{R}^d_x \times \mathbb{R}^d_{x^\prime})\right)
\end{equation}
denote the
inverse Wigner transform of $f$. Then $f$ solves
\begin{equation}
\left( \partial_t + v \cdot \nabla_x\right) f = 0
\end{equation}
in in the sense of distributions, if and only if $\gamma$
solves
\begin{equation}
\left( i \partial_t + \frac{1}{2} \left(
\Delta_x - \Delta_{x^\prime} \right) \right) \gamma (t,x,x^\prime) = 0
\end{equation}
in the sense of distributions.
\end{lemma}
\begin{proof}
Assume that
\begin{equation}
\left( \partial_t + v\cdot \nabla_x\right) f = g
\end{equation}
Using the definition of the inverse
Wigner transform
we have
\begin{equation}
\begin{aligned}
& i \partial_t \gamma (t,x,x^\prime) =
\int_{\mathbb{R}^d} i \partial_t f \left( 
t,\frac{x+x^\prime}{2},v\right) e^{i v\cdot (x-x^\prime)} dv \\
& \qquad = \int_{\mathbb{R}^d} i
\left( - v \cdot \nabla_x f + g\right)
\left( t, \frac{x+x^\prime}{2},v\right) e^{i v\cdot (x-x^\prime)} dv 
\end{aligned}
\end{equation} 
Let us focus on the transport term, $v\cdot \nabla_x f$. We have
\begin{equation}
\begin{aligned}
& \int_{\mathbb{R}^d} i \left( - v \cdot \nabla_x f\right)
\left( t,\frac{x+x^\prime}{2},v\right) e^{iv\cdot (x-x^\prime)} dv \\
& \qquad = - i \int_{\mathbb{R}^d} v \cdot
\left( \nabla_x + \nabla_{x^\prime}\right) \left[
f\left( t,\frac{x+x^\prime}{2},v\right) \right]
e^{i v \cdot (x-x^\prime)} dv \\
& \qquad = - \int_{\mathbb{R}^d} 
\left( \nabla_x + \nabla_{x^\prime}\right) \left[
f\left(t,\frac{x+x^\prime}{2},v\right)\right] \cdot
i v e^{i v \cdot (x-x^\prime)} dv \\
& \qquad = - \left( \nabla_x + \nabla_{x^\prime}\right)\cdot
\int_{\mathbb{R}^d} 
f\left(t,\frac{x+x^\prime}{2},v\right) 
i v e^{i v \cdot (x-x^\prime)} dv \\
& \qquad = - \left( \nabla_x + \nabla_{x^\prime}\right)\cdot
\int_{\mathbb{R}^d} 
f\left(t,\frac{x+x^\prime}{2},v\right) 
\frac{\nabla_x - \nabla_{x^\prime}}{2}
e^{i v \cdot (x-x^\prime)} dv \\
& \qquad = - \frac{1}{2}
\left( \nabla_x + \nabla_{x^\prime}\right)\cdot
\left( \nabla_x - \nabla_{x^\prime}\right) 
\int_{\mathbb{R}^d} 
f\left(t,\frac{x+x^\prime}{2},v\right) 
e^{i v \cdot (x-x^\prime)} dv \\
& \qquad = - \frac{1}{2}
\left( \Delta_x - \Delta_{x^\prime}\right)
\int_{\mathbb{R}^d} 
f\left(t,\frac{x+x^\prime}{2},v\right) 
e^{i v \cdot (x-x^\prime)} dv \\
& \qquad = - \frac{1}{2}
\left( \Delta_x - \Delta_{x^\prime}\right) \gamma (t,x,x^\prime) \\
\end{aligned}
\end{equation}
Therefore,
\begin{equation}
\left( i \partial_t + \frac{1}{2} \left( \Delta_x - \Delta_{x^\prime}
\right) \right) \gamma (t,x,x^\prime) =
\int_{\mathbb{R}^d} g \left( t,\frac{x+x^\prime}{2},v\right)
e^{i v \cdot(x-x^\prime)} dv
\end{equation}
But $g$ vanishes identically if and only if its inverse Wigner transform
vanishes identically. Therefore, $\gamma$ solves
(\ref{eq:LSCH}) if and only if $f$ solves (\ref{eq:FTR}).
\end{proof}

\subsection{The main result for the Boltzmann equation} 

It is possible to compute explicitly the equation satisfied by
$\gamma$ if the Wigner transform $f$ is smooth with rapid decay
and satisfies Boltzmann's
equation, (\ref{eq:boltz}); see Corollary \ref{cor:App-BE-TR}
and Corollary \ref{cor:App-BE-Coll} in Appendix \ref{sec:appendix-A}.
The result is as follows:
\begin{equation}
\label{eq:boltz-gamma}
\left( i \partial_t + \frac{1}{2} \left(
\Delta_x - \Delta_{x^\prime}\right)\right) \gamma (t)
= B \left( \gamma(t),\gamma(t)\right)
\end{equation}
\begin{equation}
B(\gamma_1,\gamma_2) = 
B^+ (\gamma_1,\gamma_2) - 
B^- (\gamma_1,\gamma_2)
\end{equation}
\begin{equation}
\begin{aligned}
& B^- (\gamma_1,\gamma_2) (x,x^\prime) = 
\frac{i}{2^{2d} \pi^d} \int_{\mathbb{S}^{d-1}} d\omega
\int_{\mathbb{R}^d} dz \hat{\mathbf{b}}^\omega
\left( \frac{z}{2}\right) \times\\
& \qquad  
\times \gamma_1\left(x-\frac{z}{4},x^\prime+\frac{z}{4}\right)
\gamma_2\left(\frac{x+x^\prime}{2}+\frac{z}{4},
\frac{x+x^\prime}{2}-\frac{z}{4}\right)
\end{aligned}
\end{equation}
\begin{equation}
\begin{aligned}
& B^+ (\gamma_1,\gamma_2) (x,x^\prime) = 
\frac{i}{2^{2d} \pi^d} \int_{\mathbb{S}^{d-1}} d\omega
\int_{\mathbb{R}^d} dz \hat{\mathbf{b}}^\omega
\left( \frac{z}{2}\right) \times\\
& \qquad  
\times \gamma_1\left(x-\frac{1}{2} P_\omega (x-x^\prime)-
\frac{R_\omega(z)}{4},
x^\prime+\frac{1}{2}P_\omega (x-x^\prime)+\frac{R_\omega (z)}{4}\right)
\times \\
& \qquad \times \gamma_2\left(\frac{x+x^\prime}{2}
+\frac{1}{2} P_\omega (x-x^\prime) +\frac{R_\omega (z)}{4},
\frac{x+x^\prime}{2}-\frac{1}{2}
P_\omega (x-x^\prime) -\frac{R_\omega (z)}{4}\right)
\end{aligned}
\end{equation}
where we define
\begin{equation}
P_\omega (x) = \left( \omega \cdot x\right) \omega
\end{equation}
\begin{equation}
R_\omega (x) = \left( \mathbb{I} - 2 P_\omega\right) (x)
\end{equation}
and $\mathbb{I} (x) = x$. 
Solutions of Boltzmann's equation (in the $\gamma$ formulation)
 are understood using Duhamel's
formula:
\begin{equation}
\gamma(t) = e^{\frac{1}{2} i t \Delta_{\pm}} \gamma(0) - i
\int_0^t e^{\frac{1}{2}
i (t-t_1) \Delta_{\pm}}  B(\gamma(t_1),\gamma(t_1)) dt_1
\end{equation}
Here $\Delta_{\pm} = \Delta_x - \Delta_{x^\prime}$.

We are now ready to state our first main result.
\begin{theorem}
\label{thm:BE-LWP}
Suppose $A\in[0,1]$, $\alpha \in \left(\frac{d-1}{2},\infty\right)$,
$\beta \in \left(d,\infty\right)$,
$\kappa \in \left( 0,\infty\right)$,
 and additionally $\frac{1}{\sigma} \in \left(\max (0,2A-1), 2\right]$;
fix any $\lambda \in\left( 0,\infty\right)$. Consider the Boltzmann equation 
(\ref{eq:boltz-gamma}) with $\mathbf{b}\in L^\infty_A$. For any
$\gamma_0 \in H^{\alpha,\beta,\sigma,\kappa}$ there exists a unique
solution $\gamma(t)$ of Boltzmann's equation on a small time interval
$[0,T]$ such that
\begin{equation}
\left\Vert \left\Vert
\gamma(t)\right\Vert_{H^{\alpha,\beta,\sigma,\kappa-\lambda t}}
\right\Vert_{L^\infty_{t\in [0,T]}} < \infty
\end{equation}
and
\begin{equation}
\left\Vert \left\Vert
B(\gamma(t),\gamma(t))\right\Vert_{H^{\alpha,\beta,\sigma,\kappa-\lambda t}}
\right\Vert_{L^1_{t\in [0,T]}} < \infty
\end{equation}
both hold, and $\gamma(0)=\gamma_0$. Moreover, for some 
$r\in [0,1)$ we have the following: if
$\left\Vert \gamma_0 \right\Vert_{H^{\alpha,\beta,\sigma,\kappa}}
\leq M$ then for all small enough $T$
depending only on  $\alpha,\beta,\kappa,\sigma,\lambda$ and $M$, 
there holds:
\begin{equation}
\label{eq:BE-est}
\begin{aligned}
& T^{\frac{1}{2} (1-r)}  \left\Vert \left\Vert \gamma(t) 
\right\Vert_{H^{\alpha,\beta,\sigma,\kappa-\lambda t}}
\right\Vert_{L^\infty_{t\in [0,T]}} +
\left\Vert \left\Vert B \left(\gamma(t),\gamma(t)\right)
\right\Vert_{H^{\alpha,\beta,\sigma,\kappa-\lambda t}}
\right\Vert_{L^1_{t\in [0,T]}} \leq \\
& \qquad \qquad \qquad \qquad \leq 
C(M,\alpha,\beta,\sigma,\kappa,\lambda)\times T^{\frac{1}{2} (1-r)} 
\left\Vert \gamma_0 \right\Vert_{H^{\alpha,\beta,\sigma,\kappa}} 
\end{aligned}
\end{equation}
If $A \in 
\left[0,\frac{1}{2} \right)$ then we may take $\lambda = 0$ and
$\kappa \in \left[ 0,\infty\right)$ and the same results hold, with the same
restrictions on $\alpha,\beta,\sigma$.
\end{theorem}
\begin{remark}
If $A=0$ it is possible to optimize the proof of Theorem \ref{thm:BE-LWP}
and obtain the same
result, with $\lambda = 0$, for any $\kappa \in \left[0,\infty\right)$,
$\frac{1}{\sigma} \in (0,2]$, 
and $\alpha,\beta \in\left( \frac{d-1}{2},\infty\right)$.
We omit the details.
\end{remark}

\subsection{The main result for the Boltzmann hierarchy}

We now turn to the Boltzmann hierarchy. The Boltzmann hierarchy is an
infinite sequence of coupled PDEs describing the evolution of
densities $f^{(k)} \left(t,x_1,\dots,x_k,v_1,\dots,v_k\right)$ for
$k\in \mathbb{N} = \left\{ 1,2,3,\dots\right\}$. The densities
$f^{(k)}$ are assumed to be symmetric with respect to interchange
of particle indices.
The Boltzmann hierarchy arises as an intermediate equation in the
derivation of Boltzmann's equation from an underlying Hamiltonian
evolution of many particles, \cite{L1975,Ki1975,GSRT2014}.
We use the notation $X_k = (x_1,\dots,x_k)$ and, for $i\leq j$,
$X_{i:j} = (x_i,x_{i+1},\dots,x_j)$, and similarly for $V_k$ and
$V_{i:j}$. 
For each $k\in\mathbb{N}$, the $k$th equation of the Boltzmann
hierarchy is written:
\begin{equation}
\label{eq:Boltz-hierarchy}
\left( \partial_t + V_k \cdot \nabla_{X_k} \right) f^{(k)} (t,X_k,V_k) =
C_{k+1} f^{(k+1)} (t,X_k,V_k)
\end{equation}
where the collision operator $C_{k+1}$ is split into gain and loss
parts:
\begin{equation}
C_{k+1} f^{(k+1)} =\sum_{i=1}^k  C^+_{i,k+1} f^{(k+1)} - 
\sum_{i=1}^k C^-_{i,k+1} f^{(k+1)}
\end{equation}
The gain term is written
\begin{equation}
\begin{aligned}
& C^+_{i,k+1} f^{(k+1)} (t,X_k,V_k) = \\ 
& = \int_{\mathbb{R}^d \times \mathbb{S}^{d-1}}
dv_{k+1} d\omega \mathbf{b}
 \left( |v_{k+1}-v_i|, \omega\cdot\frac{v_{k+1}-v_i}
{|v_{k+1}-v_i|} \right) \times \\
& \qquad \qquad \qquad \qquad \times f^{(k+1)} (t,x_1,\dots,x_i,\dots,
x_k,x_i,v_1,\dots,v_i^*,\dots,v_k,v_{k+1}^*)
\end{aligned}
\end{equation}
where
\begin{equation}
\begin{aligned}
v_i^* & = v_i + P_\omega \left( v_{k+1} - v_i \right) \\
v_{k+1}^* & = v_{k+1} - P_\omega \left( v_{k+1} - v_i \right)
\end{aligned}
\end{equation}
Similarly for the loss term we have
\begin{equation}
\begin{aligned}
& C^-_{i,k+1} f^{(k+1)} (t,X_k,V_k) = \\ 
& = \int_{\mathbb{R}^d \times \mathbb{S}^{d-1}}
dv_{k+1} d\omega \mathbf{b}
 \left( |v_{k+1}-v_i|, \omega\cdot\frac{v_{k+1}-v_i}
{|v_{k+1}-v_i|} \right) \times \\
& \qquad \qquad \qquad \qquad \times f^{(k+1)} (t,x_1,\dots,x_i,\dots,
x_k,x_i,v_1,\dots,v_i,\dots,v_k,v_{k+1})
\end{aligned}
\end{equation}
Note carefully that the collision operators $C_{i,k+1}^{\pm}$ involve
the evaluation of $f^{(k+1)}$ along the hypersurface
$x_{k+1} = x_i$.

In exactly the same manner as for the Boltzmann equation, we define the Wigner
tranform and its inverse for multiple particles:
\begin{equation}
\label{eq:inv-wigner}
\gamma^{(k)} (t,X_k,X_k^\prime) = \int_{\mathbb{R}^{dk}}
f^{(k)} \left(t,\frac{X_k+X_k^\prime}{2},V_k\right)
e^{iV_k\cdot (X_k-X_k^\prime)} dV_k
\end{equation}
\begin{equation}
f^{(k)} (t,X_k,V_k) = \frac{1}{(2\pi)^{dk}}
\int_{\mathbb{R}^{dk}} \gamma^{(k)} \left( t,X_k+\frac{Y_k}{2},
X_k-\frac{Y_k}{2}\right)
e^{-iV_k\cdot Y_k} dY_k
\end{equation}
The Fourier transform of $\gamma^{(k)}$ is written
\begin{equation}
\begin{aligned}
& \hat{\gamma}^{(k)}
 (\xi_1,\dots,\xi_k,\xi_1^\prime,\dots,\xi_k^\prime)
= \\
& \qquad \qquad \qquad \qquad 
 = \int e^{-i \sum_{i=1}^k x_i\cdot \xi_i}
e^{-i \sum_{i=1}^k x_i^\prime \cdot \xi_i^\prime}
\gamma^{(k)} (X_k,X_k^\prime) dX_k dX_k^\prime
\end{aligned}
\end{equation}
Let us define the weighted Sobolev spaces $H^{\alpha,\beta,\sigma,\kappa}_k$
for $\alpha,\beta,\kappa\geq 0$ and $\sigma > 0$:
\begin{equation}
\begin{aligned}
& \left\Vert \gamma^{(k)} (X_k,X_k^\prime) 
\right\Vert_{H^{\alpha,\beta,\sigma,\kappa}_k} =\\
&\qquad = 
\left\Vert 
\begin{aligned}
& \prod_{i=1}^k \left\{ \left< \xi_i+\xi_i^\prime\right>^\alpha
\left< \xi_i-\xi_i^\prime\right>^\beta
e^{\kappa\left< \xi_i-\xi_i^\prime\right>^{\frac{1}{\sigma}}}\right\}\times\\
& \qquad \qquad \qquad \qquad
\times \hat{\gamma}^{(k)} (\xi_1,\dots,\xi_k,\xi_1^\prime,\dots,\xi_k^\prime)
\end{aligned}
\right\Vert_{L^2_{\xi_1,\dots,\xi_k,
\xi_1^\prime,\dots,\xi_k^\prime}}
\end{aligned}
\end{equation}
These norms are equivalent (up to a factor like $C^k$) to the following norms
for classical densities:
\begin{equation}
\begin{aligned}
 \left\Vert \prod_{i=1}^k \left\{ \left< 2v_i\right>^\beta 
e^{\kappa \left<2v_i\right>^{\frac{1}{\sigma}}}
\left( 1-\Delta_{x_i}\right)^{\frac{\alpha}{2}}\right\} f^{(k)}(X_k,V_k)
\right\Vert_{L^2_{X_k,V_k}}
\end{aligned}
\end{equation}
If $\Gamma = \left\{ \gamma^{(k)} \right\}_{k\in\mathbb{N}}$ and
$\xi > 0$ then we further define
\begin{equation}
\left\Vert \Gamma \right\Vert_{\mathcal{H}^{\alpha,\beta,\sigma,\kappa}_\xi}
= \sum_{k\in\mathbb{N}} \xi^{k} \left\Vert
\gamma^{(k)} (X_k,V_k)\right\Vert_{H^{\alpha,\beta,\sigma,\kappa}_k}
\end{equation}
Note that $\gamma \in H^{\alpha,\beta,\sigma,\kappa}$ if and only
if $\Gamma = \left\{ \gamma^{\otimes k}\right\}_{k\in\mathbb{N}} \in
\mathcal{H}^{\alpha,\beta,\sigma,\kappa}_\xi$ for some (arbitrary)
$\xi > 0$. 

The inverse Wigner transform of the Boltzmann hierarchy is:
(see Proposition \ref{prop:App-BH-TR} and Proposition
\ref{prop:App-BH-Coll} in Appendix \ref{sec:appendix-A})
\begin{equation}
\label{eq:BH-gamma}
\begin{aligned}
& \left(i \partial_t + \frac{1}{2} \left( \Delta_{X_k} - 
\Delta_{X_k^\prime}\right)
\right) \gamma^{(k)} (t,X_k,X_k^\prime) = B_{k+1} \gamma^{(k+1)} 
(t,X_k,X_k^\prime)
\end{aligned}
\end{equation}
\begin{equation}
\begin{aligned}
B_{k+1} \gamma^{(k+1)}
 = \sum_{i=1}^k \left( B_{i,k+1}^+ \gamma^{(k+1)}
 - B_{i,k+1}^- \gamma^{(k+1)}\right) 
\end{aligned}
\end{equation}
\begin{equation}
\begin{aligned}
& B_{i,k+1}^- \gamma^{(k+1)} (t,X_k,X_k^\prime) = \\
& = \frac{i}{2^{2d} \pi^d} \int_{\mathbb{S}^{d-1}} d\omega
\int_{\mathbb{R}^d}  dz \;
 \hat{\mathbf{b}}^\omega \left(\frac{z}{2}\right)\times \\
& \times \gamma^{(k+1)} \left( t,X_{1:(i-1)},x_i-\frac{z}{4},X_{(i+1):k},
\frac{x_i+x_i^\prime}{2}+\frac{z}{4}, \right. \\
& \qquad \qquad \qquad \qquad \qquad 
\left. X_{1:(i-1)}^\prime, x_i^\prime + \frac{z}{4},X_{(i+1):k}^\prime,
\frac{x_i+x_i^\prime}{2}-\frac{z}{4} \right) \\
\end{aligned}
\end{equation}
\begin{equation}
\begin{aligned}
& B_{i,k+1}^+ \gamma^{(k+1)} (t,X_k,X_k^\prime) = \\
& = \frac{i}{2^{2d} \pi^d} 
\int_{\mathbb{S}^{d-1}} d\omega
\int_{\mathbb{R}^d} dz \; \hat{\mathbf{b}}^\omega
\left( \frac{z}{2}\right) \times \\
& \times  \gamma^{(k+1)} \left(
t,X_{1:(i-1)},x_i-\frac{1}{2} P_\omega
(x_i-x_i^\prime)-\frac{R_\omega (z)}{4}, X_{(i+1):k}, \right. \\
& \qquad \qquad \qquad \qquad \qquad \qquad \left.
\frac{x_i+x_i^\prime}{2}+\frac{1}{2} P_\omega
(x_i-x_i^\prime)+\frac{R_\omega (z)}{4},\right.\\
& \qquad \qquad  \left. 
X_{1:(i-1)}^\prime, x_i^\prime + \frac{1}{2}P_\omega (x_i-x_i^\prime)+
\frac{R_\omega (z)}{4}, X_{(i+1):k}^\prime, \right. \\
& \qquad \qquad \qquad \qquad \qquad \qquad \left.
\frac{x_i+x_i^\prime}{2}-\frac{1}{2}P_\omega
(x_i-x_i^\prime)-\frac{R_\omega (z)}{4}\right)  \\
\end{aligned}
\end{equation}
Solutions of the Boltzmann hierarchy are understood using
Duhamel's formula: for all $k\in\mathbb{N}$,
\begin{equation}
\gamma^{(k)} (t) = e^{\frac{1}{2} i t \Delta_{\pm}^{(k)}} 
\gamma^{(k)} (0) - i
\int_0^t e^{\frac{1}{2} i (t-t_1) \Delta_{\pm}^{(k)}}
B_{k+1} \gamma^{(k+1)} (t_1) dt_1
\end{equation}
Here $\Delta_{\pm}^{(k)} = \Delta_{X_k} - \Delta_{X_k^\prime}$.
We further define $B\Gamma = \left\{ B_{k+1} \gamma^{(k+1)}
\right\}_{k\in\mathbb{N}}$.

We are ready to state our second main result.
\begin{theorem}
\label{thm:BH-LWP}
Suppose $\mathbf{b}\in L^\infty_A$ with 
$A \in \left[ 0,\frac{1}{2}\right)$, $\alpha \in
\left( \frac{d-1}{2},\infty\right)$, $\beta \in (d,\infty)$,
$\kappa \in [0,\infty)$, and
$\frac{1}{\sigma} \in \left( 0,2\right]$.
Assume $\Gamma_0 = \left\{ \gamma^{(k)}_0\right\}_{k\in\mathbb{N}} \in
\mathcal{H}^{\alpha,\beta,\sigma,\kappa}_{\xi_1}$ where
$\xi_1 \in (0,1)$, and further assume that the functions
$\gamma^{(k)}_0$ are symmetric under particle interchange. 
Then there exists $T>0$ and $0 < \xi_2 < \xi_1$ such that there
exists a unique solution $\Gamma(t)$ of the Boltzmann hierarchy
(\ref{eq:BH-gamma}) for $t\in [0,T]$ with
$\left\Vert \Gamma(t) \right\Vert_{L^\infty_{t\in [0,T]}
\mathcal{H}^{\alpha,\beta,\sigma,\kappa}_{\xi_2}} < \infty$
and
$\left\Vert B \Gamma(t) \right\Vert_{L^1_{t\in [0,T]}
\mathcal{H}^{\alpha,\beta,\sigma,\kappa}_{\xi_2}} < \infty$,
and $\Gamma (0) = \Gamma_0$.
Moreover, the following estimate holds:
\begin{equation}
\left\Vert \Gamma
\right\Vert_{L^\infty_{t\in [0,T]}
\mathcal{H}^{\alpha,\beta,\sigma,\kappa}_{\xi_2}} +
\left\Vert B \Gamma 
\right\Vert_{L^1_{t\in [0,T]}
\mathcal{H}^{\alpha,\beta,\sigma,\kappa}_{\xi_2}} \leq
C
\left\Vert \Gamma_0
\right\Vert_{\mathcal{H}^{\alpha,\beta,\sigma,\kappa}_{\xi_1}}
\end{equation}
where $C$ depends on
$T, d, \xi_1, \xi_2, \alpha, \beta, \sigma, \kappa$.
\end{theorem}

\begin{remark}
If $A=0$ it is possible to optimize the proof of Theorem \ref{thm:BH-LWP}
and obtain the same
result for any $\kappa \in [0,\infty)$, $\frac{1}{\sigma} \in (0,2]$, and
$\alpha,\beta \in \left( \frac{d-1}{2},\infty\right)$. We omit
the details.
\end{remark}

\subsection{Interpretation of the Boltzmann hierarchy}

Extending Theorem \ref{thm:BH-LWP} to the full range $A\in [0,1]$
would require revising the
boardgame argument as presented in \cite{KM2008} to be compatible with
time-dependent weights, as in Theorem \ref{thm:BE-LWP}. Unfortunately
this seems to be technically out of reach at the present time; indeed,
it seems to be an interesting open question to determine whether the
hard sphere Boltzmann hierarchy is in fact locally well-posed for data
$\Gamma (0) \in \mathcal{H}_\xi^{\alpha,\beta,\sigma,\kappa}$ with
a suitable choice of parameters.

Since we cannot (at present) extend our well-posedness result to
the hard sphere Boltzmann hierarchy
 ($A=1$), the reader will rightfully question why
we study the Boltzmann hierarchy at all. After all, the 
hard sphere interaction is
the \emph{only} interaction with Grad cut-off that is physically
relevant (and all our results assume the Grad cut-off). In particular,
at present, we have nothing to offer in the context of Lanford's theorem,
even at the level of the Boltzmann hierarchy.
Nevertheless, the Boltzmann hierarchy always has an interpretation
in the context of \emph{statistical solutions} of the Boltzmann equation.
(See \cite{CIP1994} for a formal discussion of statistical solutions.)
 Under suitable regularity assumptions,
if $\pi_t$ is a statistical solution
of Boltzmann's equation, then
\begin{equation}
f^{(k)} (t) = \int h^{\otimes k} d\pi_t (h)
\end{equation}
is a solution of the Boltzmann hierarchy (for any interaction, physical
or not). Conversely, suppose  the functions
$f^{(k)} (t)$ (assumed smooth and growing at most exponentially in $k$),
which solve the Boltzmann hierarchy,
 define the joint distribution of some
 exchangeable sequence of random variables
$(x_1,v_1),(x_2,v_2),\dots$. In that case,
 the Hewitt-Savage theorem
guarantees the existence of a unique underlying
$\pi_t$ which must be a statistical solution of Boltzmann's equation. 
\cite{HS1955}

\section{The key proposition} 
\label{sec:key}

The proofs of Theorems \ref{thm:BE-LWP} and \ref{thm:BH-LWP} will
rely on the following proposition:

\begin{proposition}
\label{prop:spacetime-est}
Suppose $A \in [0,1]$,
$\alpha \in \left(\frac{d-1}{2},\infty\right)$,
$\beta \in (d,\infty)$, and 
$\frac{1}{\sigma} \in \left( \max(0,2A-1),2\right]$. Then for
any $r\in [0,1)$ such that $\frac{r}{\sigma} \geq
\max (0,2A-1+\delta)$ for a small $\delta > 0$ we have for all
$\kappa_0 > \kappa > 0$, any $1\leq i \leq k$, and any
$\gamma_0^{(k+1)} \in H^{\alpha,\beta,\sigma,\kappa_0}_{k+1}$
 the following estimates:
\begin{equation}
\label{eq:spacetime-est-1}
\begin{aligned}
& \left\Vert B^{\pm}_{i,k+1} \left[ 
e^{\frac{1}{2} i t \left( \Delta_{X_{k+1}} - \Delta_{X_{k+1}^\prime}\right)}
\gamma_0^{(k+1)} \right] \right\Vert_{L^2_t 
H^{\alpha,\beta,\sigma,\kappa}_k } \leq \\
& \qquad \qquad \qquad   \leq
C(\alpha,\beta,\sigma,r) \left\Vert \mathbf{b} \right\Vert_{L^\infty_A}
\left( 1+\left(\kappa_0-\kappa\right)^{-\frac{1}{2}r}\right)
 \left\Vert \gamma_0^{(k+1)} 
\right\Vert_{H^{\alpha,\beta,\sigma,\kappa_0}_{k+1}}
\end{aligned}
\end{equation}
Moreover, if $A\in \left[ 0,\frac{1}{2}\right)$,
$\alpha \in \left(\frac{d-1}{2},\infty\right)$,
$\beta \in (d,\infty)$, and
$\frac{1}{\sigma} \in (0,2]$, then for any
$\kappa_0 \geq \kappa \geq 0$, any $1\leq i \leq k$, and any
$\gamma_0^{(k+1)} \in H^{\alpha,\beta,\sigma,\kappa_0}_{k+1}$
the following estimates hold:
\begin{equation}
\label{eq:spacetime-est-2}
\begin{aligned}
& \left\Vert B^{\pm}_{i,k+1} \left[ 
e^{\frac{1}{2} i t \left( \Delta_{X_{k+1}} - \Delta_{X_{k+1}^\prime}\right)}
\gamma_0^{(k+1)} \right] \right\Vert_{L^2_t 
H^{\alpha,\beta,\sigma,\kappa}_k } \leq \\
& \qquad \qquad \qquad\qquad \qquad \qquad \leq
C(\alpha,\beta,\sigma) \left\Vert \mathbf{b} \right\Vert_{L^\infty_A}
 \left\Vert \gamma_0^{(k+1)} 
\right\Vert_{H^{\alpha,\beta,\sigma,\kappa_0}_{k+1}}
\end{aligned}
\end{equation}
\end{proposition}
\begin{remark}
Note that the second part of Proposition \ref{prop:spacetime-est} 
formally follows from the first part by setting $r=0$. In fact we
will only prove the first part since the second part follows after
trivial changes to the proof.
\end{remark}
\begin{remark}
If $A=0$ it is possible to optimize the proof of Proposition
\ref{prop:spacetime-est} and obtain (\ref{eq:spacetime-est-2})
for any $\kappa_0 \geq \kappa \geq 0$, $\frac{1}{\sigma} \in
(0,2]$, and
$\alpha,\beta \in \left( \frac{d-1}{2},\infty\right)$. We omit
the details.
\end{remark}

\subsection*{Proof of Proposition \ref{prop:spacetime-est}}

\subsection*{Loss Term}

Consider a typical part of the loss term, e.g.
 $B_{1,k+1}^- \gamma^{(k+1)}$:
\begin{equation}
\begin{aligned}
& B_{1,k+1}^- \gamma^{(k+1)} (t,X_k,X_k^\prime) = \frac{i}{2^{2d} \pi^d}
 \int_{\mathbb{S}^{d-1}} d\omega \int_{\mathbb{R}^d} dz
\hat{\mathbf{b}}^\omega \left( \frac{z}{2}\right) \times \\
& \qquad \times \gamma^{(k+1)} \left( t,x_1-\frac{z}{4},X_{2:k},
\frac{x_1+x_1^\prime}{2}+\frac{z}{4},
x_1^\prime + \frac{z}{4},
X_{2:k}^\prime,\frac{x_1+x_1^\prime}{2}-\frac{z}{4}\right)
\end{aligned}
\end{equation}
We will fix some initial data $\gamma_0^{(k+1)} (X_{k+1},X_{k+1}^\prime)$ and
consider the following function:
\begin{equation}
B^-_{1,k+1}
 \left[ e^{\frac{1}{2} i t \left( \Delta_{X_{k+1}} - 
\Delta_{X_{k+1}^\prime}\right)}
\gamma_0^{(k+1)} \right] (t,X_k,X_k^\prime)
\end{equation}
The spacetime Fourier transform of a function $F (t,X_k,X_k^\prime)$ is
\begin{equation}
\begin{aligned}
& \tilde{F} (\tau,\xi_1,\dots,\xi_k,\xi_1^\prime,\dots,\xi_k^\prime) =\\
&\qquad \qquad \qquad = \int dt dX_k dX_k^\prime e^{-i t \tau}
e^{-i \sum_{i=1}^k x_i \cdot \xi_i} 
e^{-i \sum_{i=1}^k x_i^\prime \cdot \xi_i^\prime}
F(t,X_k,X_k^\prime)
\end{aligned}
\end{equation}
The spacetime Fourier transform of 
$e^{\frac{1}{2} i t (\Delta_{X_{k+1}} - \Delta_{X_{k+1}^\prime})} 
\gamma_0^{(k+1)}$ is, up to a constant depending on $k$,
\begin{equation}
\hat{\gamma}_0^{(k+1)}
 ( \xi_1,\dots,\xi_{k+1},\xi_1^\prime,\dots,\xi_{k+1}^\prime) 
\delta \left( \tau + \frac{1}{2} \sum_{i=1}^{k+1}
|\xi_i|^2 - \frac{1}{2} \sum_{i=1}^{k+1} |\xi_i^\prime|^2\right)
\end{equation}
We also have
\begin{equation}
\begin{aligned}
&\left( B^-_{1,k+1}\left[ e^{\frac{1}{2} i t \left( \Delta_{X_{k+1}} -
 \Delta_{X_{k+1}^\prime}
\right)} \gamma_0^{(k+1)}\right] \right)^{\sim} (\tau,\xi_1,\dots,
\xi_k,\xi^\prime_1,\dots,\xi_k^\prime) = \\
& = \textnormal{cst.} \int d\eta d\eta^\prime 
\delta \left( \tau + \frac{1}{2} \left| \xi_1 - \frac{\eta + \eta^\prime}{2}
\right|^2 - \frac{1}{2} \left| \xi_1^\prime - 
\frac{\eta + \eta^\prime}{2}\right|^2 +\right. \\
&\qquad \qquad \qquad \qquad \qquad 
 \left. + \frac{1}{2} |\eta|^2  - \frac{1}{2} |\eta^\prime|^2
+ \frac{1}{2}\sum_{2\leq i \leq k} (|\xi_i|^2-|\xi_i^\prime|^2)
 \right)\times \\
& \times \mathbf{b} \left( \frac{|-\xi_1+\xi_1^\prime+\eta-\eta^\prime|}{2},
\omega\cdot\frac{-\xi_1+\xi_1^\prime+\eta-\eta^\prime}
{|-\xi_1+\xi_1^\prime+\eta-\eta^\prime|}\right) \times \\
& \times \hat{\gamma}_0^{(k+1)} \left(
\xi_1 - \frac{\eta+\eta^\prime}{2},\xi_2,\dots,\xi_k,\eta,
\xi_1^\prime - \frac{\eta+\eta^\prime}{2},\xi_2^\prime,\dots,
\xi_k^\prime,\eta^\prime\right)
\end{aligned}
\end{equation}
The constant is uniformly bounded in $k$. Now we simply bound the collision
kernel $\mathbf{b}$ using 
$\left\Vert \mathbf{b}\right\Vert_{L^\infty_A}$ to yield:
\begin{equation}
\begin{aligned}
&\left|\left( B^-_{1,k+1}\left[ e^{\frac{1}{2} i t \left( \Delta_{X_{k+1}} -
 \Delta_{X_{k+1}^\prime}
\right)} \gamma_0^{(k+1)}\right] \right)^{\sim} (\tau,\xi_1,\dots,
\xi_k,\xi^\prime_1,\dots,\xi_k^\prime) \right| \lesssim \\
& \lesssim \left\Vert \mathbf{b}
 \right\Vert_{L^\infty_A} \int d\eta d\eta^\prime 
\left( \left< \xi_1-\xi_1^\prime\right>^A +
\left< \eta-\eta^\prime\right>^A \right)\times \\
& \times
\delta \left( \tau + \frac{1}{2} \left| \xi_1 - \frac{\eta + \eta^\prime}{2}
\right|^2 - \frac{1}{2} \left| \xi_1^\prime - 
\frac{\eta + \eta^\prime}{2}\right|^2 +\right. \\
&\qquad \qquad \qquad \qquad \qquad 
 \left. + \frac{1}{2} |\eta|^2  - \frac{1}{2} |\eta^\prime|^2
+ \frac{1}{2}\sum_{2\leq i \leq k} (|\xi_i|^2-|\xi_i^\prime|^2)
 \right)\times \\
& \times \left| \hat{\gamma}_0^{(k+1)} \left(
\xi_1 - \frac{\eta+\eta^\prime}{2},\xi_2,\dots,\xi_k,\eta,
\xi_1^\prime - \frac{\eta+\eta^\prime}{2},\xi_2^\prime,\dots,
\xi_k^\prime,\eta^\prime\right)\right|
\end{aligned}
\end{equation}

We want to estimate the following integral, for suitable 
$\alpha,\beta,\kappa,\sigma > 0$:
\begin{equation}
\begin{aligned}
I^- (\alpha,\beta,\kappa,\sigma) & = \int d\tau d\xi_1\dots d\xi_k
d\xi_1^\prime \dots d\xi_k^\prime \times \\
& \qquad \qquad \qquad \times \prod_{i=1}^k
\left\{ \left< \xi_i + \xi_i^\prime \right>^{2\alpha}
\left< \xi_i - \xi_i^\prime \right>^{2\beta}
e^{2\kappa \left< \xi_i-\xi_i^\prime\right>^{\frac{1}{\sigma}}}
\right\}\times \\
& \qquad \qquad \qquad \times 
\left|\left(B_{1,k+1}^-
\left[ e^{\frac{1}{2} i t \left( \Delta_{X_{k+1}}- 
\Delta_{X_{k+1}^\prime}\right)}
\gamma_0^{(k+1)} \right] \right)^\sim \right|^2
\end{aligned}
\end{equation}
To start, observe that
\begin{equation*}
\begin{aligned}
& I^- (\alpha,\beta,\kappa,\sigma) \lesssim 
\left\Vert \mathbf{b} \right\Vert_{L^\infty_A}^2
\int d\tau d\xi_1\dots d\xi_k
d\xi_1^\prime\dots d\xi_k^\prime 
d\eta_1 d\eta_1^\prime d\eta_2 d\eta_2^\prime \times \\
& \times \prod_{i=1}^k \left\{
\left< \xi_i+\xi_i^\prime \right>^{2\alpha}
\left< \xi_i-\xi_i^\prime \right>^{2\beta}
e^{2\kappa \left< \xi_i - \xi_i^\prime\right>^{\frac{1}{\sigma}}}
\right\} \times \\
& \times \left( \left< \xi_1 - \xi_1^\prime\right>^A + \left< \eta_1 - 
\eta_1^\prime\right>^A\right) \left( \left< \xi_1 - \xi_1^\prime\right>^A
+ \left< \eta_2 - \eta_2^\prime\right>^A\right)\times \\
& \times \delta \left( \tau + \frac{1}{2} \left| \xi_1 -
\frac{\eta_1+\eta_1^\prime}{2}\right|^2 - \frac{1}{2}
\left| \xi_1^\prime - \frac{\eta_1+\eta_1^\prime}{2}\right|^2+\right.\\
& \left. \qquad \qquad \qquad \qquad \qquad
+\frac{1}{2}\left|\eta_1\right|^2 - \frac{1}{2}\left|\eta_1^\prime\right|^2
+\frac{1}{2} \sum_{2\leq i \leq k} (|\xi_i|^2-|\xi_i^\prime|^2)
\right)\times \\
& \times \delta \left(  \tau + \frac{1}{2} \left| \xi_1 -
\frac{\eta_2+\eta_2^\prime}{2}\right|^2 - \frac{1}{2}
\left| \xi_1^\prime - \frac{\eta_2+\eta_2^\prime}{2}\right|^2 +\right. \\
& \qquad \qquad \qquad \qquad \qquad \left.
+\frac{1}{2}\left|\eta_2\right|^2 - \frac{1}{2}\left|\eta_2^\prime\right|^2
+\frac{1}{2} \sum_{2\leq i \leq k} (|\xi_i|^2-|\xi_i^\prime|^2)
\right)\times \\
& \times \left| \hat{\gamma}_0^{(k+1)}
\left(\xi_1-\frac{\eta_1+\eta_1^\prime}{2},\xi_2,\dots,\xi_k,\eta_1,
\xi_1^\prime-\frac{\eta_1+\eta_1^\prime}{2},\xi_2^\prime,\dots,
\xi_k^\prime,\eta_1^\prime\right) \right| \times\\
& \times \left| \hat{\gamma}_0^{(k+1)}
\left(\xi_1-\frac{\eta_2+\eta_2^\prime}{2},\xi_2,\dots,\xi_k,\eta_2,
\xi_1^\prime-\frac{\eta_2+\eta_2^\prime}{2},\xi_2^\prime,\dots,
\xi_k^\prime,\eta_2^\prime\right) \right|
\end{aligned}
\end{equation*}
Let $\kappa_0 >\kappa$, then multiply and divide the integrand by the following factor:
\begin{equation}
\prod_{j=1}^2 \left\{
\begin{aligned}
& \left< \xi_1+\xi_1^\prime-\eta_j-\eta_j^\prime \right>^\alpha
\left< \xi_1-\xi_1^\prime \right>^\beta
e^{\kappa_0 \left< \xi_1-\xi_1^\prime\right>^\frac{1}{\sigma}}\times \\
& \qquad \qquad\qquad \qquad \qquad \times\left<\eta_j+\eta_j^\prime\right>^\alpha 
\left< \eta_j - \eta_j^\prime\right>^\beta
e^{\kappa_0 \left< \eta_j-\eta_j^\prime\right>^\frac{1}{\sigma}}
\end{aligned}
\right\}
\end{equation}
Then group terms together and apply Cauchy-Schwarz \emph{pointwise}
under the integral sign. We obtain two different terms that are equal
due to symmetry under re-labeling coordinates; hence,
\begin{equation*}
\begin{aligned}
& I^- \lesssim \left\Vert \mathbf{b} \right\Vert_{L^\infty_A}^2
 \int d\tau d\xi_1 \dots d\xi_k
d\xi_1^\prime \dots d\xi_k^\prime 
d\eta_1 d\eta_1^\prime d\eta_2 d\eta_2^\prime \times \\
& \times \frac{\left<\xi_1+\xi_1^\prime\right>^{2\alpha} 
\left<\xi_1-\xi_1^\prime\right>^{2\beta}
\left( \left< \xi_1-\xi_1^\prime\right>^{2A} +
\left< \eta_1 - \eta_1^\prime\right>^{2A} \right)}
{\left<\xi_1+\xi_1^\prime-\eta_1-\eta_1^\prime\right>^{2\alpha}
\left<\xi_1-\xi_1^\prime\right>^{2\beta}
\left<\eta_1+\eta_1^\prime\right>^{2\alpha} 
\left<\eta_1-\eta_1^\prime\right>^{2\beta}} \times \\
&\times\frac{e^{2\kappa \left< \xi_1-\xi_1^\prime\right>^{\frac{1}{\sigma}}}}
{e^{2\kappa_0 \left< \xi_1-\xi_1^\prime \right>^{\frac{1}{\sigma}}}
e^{2\kappa_0 \left< \eta_1 - \eta_1^\prime\right>^{\frac{1}{\sigma}}}}\times\\
& \times \delta \left( \tau + \frac{1}{2} \left| \xi_1 -
\frac{\eta_1+\eta_1^\prime}{2}\right|^2 - \frac{1}{2}
\left| \xi_1^\prime - \frac{\eta_1+\eta_1^\prime}{2}\right|^2 + \right. \\
& \qquad \qquad \qquad \qquad \qquad\left. +\frac{1}{2}\left|\eta_1\right|^2 - 
\frac{1}{2}\left|\eta_1^\prime\right|^2 + \frac{1}{2}\sum_{2\leq i \leq k}
(|\xi_i|^2-|\xi_i^\prime|^2)
\right)\times \\
& \times \delta \left(  \tau + \frac{1}{2} \left| \xi_1 -
\frac{\eta_2+\eta_2^\prime}{2}\right|^2 - \frac{1}{2}
\left| \xi_1^\prime - \frac{\eta_2+\eta_2^\prime}{2}\right|^2 + \right.\\
& \qquad \qquad \qquad \qquad \qquad \left.
+\frac{1}{2}\left|\eta_2\right|^2 - \frac{1}{2}\left|\eta_2^\prime\right|^2
+\frac{1}{2} \sum_{2\leq i \leq k} (|\xi_i|^2-|\xi_i^\prime|^2)
\right)\times \\
& \times \left< \xi_1 + \xi_1^\prime - \eta_2 - \eta_2^\prime\right>^{2\alpha}
\left< \xi_1-\xi_1^\prime\right>^{2\beta} 
e^{2\kappa_0 \left< \xi_1-\xi_1^\prime\right>^{\frac{1}{\sigma}}} \times \\
& \times \left<\eta_2+\eta_2^\prime
\right>^{2\alpha} \left<\eta_2-\eta_2^\prime\right>^{2\beta}
e^{2\kappa_0 \left< \eta_2-\eta_2^\prime\right>^{\frac{1}{\sigma}}} \times\\
& \times \prod_{2\leq i \leq k} \left\{ \left< \xi_i + \xi_i^\prime
\right>^{2\alpha} \left<\xi_i-\xi_i^\prime\right>^{2\beta}
e^{2\kappa_0 \left< \xi_i - \xi_i^\prime\right>^{\frac{1}{\sigma}}}\right\}
\times\\
& \times\left| \hat{\gamma}_0^{(k+1)} \left(
\xi_1-\frac{\eta_2+\eta_2^\prime}{2},\xi_2,\dots,\xi_k,\eta_2,
\xi_1^\prime-\frac{\eta_2+\eta_2^\prime}{2},\xi_2^\prime,\dots,
\xi_k^\prime,\eta_2^\prime\right)\right|^2
\end{aligned}
\end{equation*}
The integral completely factorizes in the following way:
\begin{equation*}
\begin{aligned}
I^- & \leq \int d\tau d\xi_1\dots d\xi_k d\xi_1^\prime
\dots d\xi_k^\prime
\left( \int d\eta_1 d\eta_1^\prime  \dots\right)
\left( \int d\eta_2 d\eta_2^\prime \dots \right) \\
& \leq \left( \sup_{\tau,\xi_i,\xi_i^\prime} \int d\eta_1 d\eta_1^\prime
\dots \right) 
\times \int d\tau d\xi_1\dots d\xi_k d\xi_1^\prime \dots d\xi_k^\prime
\left( \int d\eta_2 d\eta_2^\prime \dots \right)
\end{aligned}
\end{equation*}
Finally we are able to conclude that if the following integral,
\begin{equation}
\label{eq:loss-integral}
\begin{aligned}
& \int d\eta d\eta^\prime \delta \left(\tau + 
\frac{1}{2} \left| \xi_1 - \frac{\eta+\eta^\prime}{2}\right|^2 -
\frac{1}{2} \left| \xi_1^\prime - \frac{\eta+\eta^\prime}{2}\right|^2+\right.\\
& \qquad\qquad \qquad\qquad \qquad \left.
+ \frac{1}{2} |\eta|^2 - \frac{1}{2} |\eta^\prime|^2 
+ \frac{1}{2} \sum_{2\leq i \leq k} (|\xi_i|^2 - |\xi_i^\prime|^2)
 \right)\times \\
& \qquad \qquad 
\times \frac{\left<\xi_1+\xi_1^\prime\right>^{2\alpha}
\left( \left<\xi_1-\xi_1^\prime\right>^{2A} +
\left<\eta-\eta^\prime\right>^{2A} \right) }
{\left<\xi_1+\xi_1^\prime-\eta-\eta^\prime\right>^{2\alpha}
\left< \eta +\eta^\prime\right>^{2\alpha} 
\left<\eta-\eta^\prime\right>^{2\beta}}
e^{-2(\kappa_0-\kappa)
\left<\xi_1-\xi_1^\prime\right>^{\frac{1}{\sigma}}}
\end{aligned}
\end{equation}
is bounded \emph{uniformly} with respect to $\tau, \xi_1,\dots,\xi_k,
\xi_1^\prime,\dots, \xi_k^\prime$, then
the following estimate holds:
\begin{equation}
\begin{aligned}
& \left\Vert B^-_{1,k+1} \left[ 
e^{\frac{1}{2} i t \left( \Delta_{X_{k+1}} - \Delta_{X_{k+1}^\prime}\right)}
\gamma_0^{(k+1)} \right] \right\Vert_{L^2_t 
H^{\alpha,\beta,\sigma,\kappa}_k } \leq \\
& \qquad \qquad \qquad \qquad \qquad \qquad \leq
C(\alpha,\beta,\sigma,\kappa,\kappa_0) \left\Vert \mathbf{b}
\right\Vert_{L^\infty_A}
 \left\Vert \gamma_0^{(k+1)} 
\right\Vert_{H^{\alpha,\beta,\sigma,\kappa_0}_{k+1}}
\end{aligned}
\end{equation}

Let us make the change of variables $w=\frac{\eta+\eta^\prime}{2}$,
$z=\frac{\eta-\eta^\prime}{2}$ in (\ref{eq:loss-integral}); then,
up to a constant, the integral becomes:
\begin{equation}
\label{eq:loss-integral-2}
\begin{aligned}
& \int dw dz \delta \left(\tau + 
\frac{1}{2} \left| \xi_1 - w\right|^2 -
\frac{1}{2} \left| \xi_1^\prime - w\right|^2
+ \frac{1}{2} |w+z|^2 - \frac{1}{2} |w-z|^2 + \right. \\
& \qquad \qquad \qquad \qquad \qquad \qquad \qquad \qquad \left.
+ \frac{1}{2} \sum_{2\leq i \leq k} \left( |\xi_i|^2-|\xi_i^\prime|^2
\right)\right)\times \\
& \qquad \qquad \qquad \qquad 
\times \frac{\left<\xi_1+\xi_1^\prime\right>^{2\alpha}
\left( \left< \xi_1-\xi_1^\prime\right>^{2A} + 
\left< 2z \right>^{2A} \right) }
{\left<\xi_1+\xi_1^\prime-2w\right>^{2\alpha}
e^{2(\kappa_0-\kappa) \left< \xi_1-\xi_1^\prime\right>^{\frac{1}{\sigma}}}
\left< 2w \right>^{2\alpha} \left<2z\right>^{2\beta}}
\end{aligned}
\end{equation}
This is the same as:
\begin{equation}
\label{eq:loss-integral-2}
\begin{aligned}
&K = \int dw dz \delta \left(\tau + 
\frac{1}{2} \sum_{i=1}^k \left( |\xi_i|^2 - |\xi_i^\prime|^2\right) -
\left( \xi_1 - \xi_1^\prime - 2z\right)\cdot w \right)\times \\
& \qquad \qquad \qquad 
\times \frac{\left<\xi_1+\xi_1^\prime\right>^{2\alpha} 
\left( \left< \xi_1-\xi_1^\prime\right>^{2A} +
\left< 2z \right>^{2A}\right)}
{\left<\xi_1+\xi_1^\prime-2w\right>^{2\alpha}
e^{2(\kappa_0-\kappa)\left<\xi_1-\xi_1^\prime\right>^{\frac{1}{\sigma}}}
\left< 2w \right>^{2\alpha} \left<2z\right>^{2\beta}}
\end{aligned}
\end{equation}
Hence, one way to parametrize the integral is to let $z\in\mathbb{R}^d$
be arbitrary and let $w$ range over a codimension one hyperplane
in $\mathbb{R}^d$; the hyperplane is determined by $\tau,\xi,\xi^\prime,z$.
We have:
\begin{equation}
K \leq \int_{\mathbb{R}^d}  dz 
\frac{  \left<\xi_1-\xi_1^\prime\right>^{2A} + 
\left<2z\right>^{2A} }
{\left|\xi_1-\xi_1^\prime - 2 z\right|
e^{2 (\kappa_0 - \kappa) \left<\xi_1-\xi_1^\prime\right>^{\frac{1}{\sigma}}}
\left< 2z \right>^{2\beta}}
\int_P \frac{dS (w) \left< \xi_1+\xi_1^\prime\right>^{2\alpha}}
{\left< \xi_1+\xi_1^\prime-2w\right>^{2\alpha} \left< 2w \right>^{2\alpha}}
\end{equation}
where $dS(w)$ is the induced surface measure on a 
hyperplane $P\subset \mathbb{R}^d$, given explicitly by
\begin{equation}
P = \left\{ w \in \mathbb{R}^d \left|
\tau + \frac{1}{2} \sum_{i=1}^k \left( |\xi_i|^2-|\xi_i^\prime|^2\right) -
\left( \xi_1-\xi_1^\prime-2z\right)\cdot w = 0\right. \right\}
\end{equation}
In order to show the uniform boundedness of $K$ with respect
to $\tau,\xi_1,\dots,\xi_k$, $\xi_1^\prime,\dots,\xi_k^\prime$,
it suffices to prove the uniform boundedness
of the following three quantities with respect to $W\in \mathbb{R}^d$:
\begin{equation}
I_1 = \sup_{P\subset \mathbb{R}^d : \dim P = d-1}
\int_P dS(w) \frac{\left<W\right>^{2\alpha}}
{\left< W - w\right>^{2\alpha} \left< w \right>^{2\alpha}}
\end{equation}
\begin{equation}
I_2 = \int_{\mathbb{R}^d} dz
\frac{1}{|W-z| \left< z \right>^{2\beta - 2A}}
\end{equation}
\begin{equation}
I_3 = \int_{\mathbb{R}^d} dz \frac{\left< W \right>^{2A}}
{e^{2(\kappa_0-\kappa)\left<W\right>^{\frac{1}{\sigma}}}
|W-z|\left<z\right>^{2\beta}}
\end{equation}
Note that in the expression for $I_1$, $P$ is an arbitrary hyperplane
of codimension one in $\mathbb{R}^d$.

We begin with $I_3$; clearly the integral over the set
$|z-W|<1$ is uniformly bounded in $W$ if $\beta \geq A$. Therefore it
suffices to bound the following integral:
\begin{equation}
I_3^\prime = \int_{\mathbb{R}^d} dz \frac{\left< W \right>^{2A} }
{e^{2(\kappa_0-\kappa)\left<W\right>^{\frac{1}{\sigma}}}
\left<W-z\right>\left<z\right>^{2\beta}}
\end{equation}
We have the following inequality:
\begin{equation}
\begin{aligned}
e^{2(\kappa-\kappa_0) \left< W \right>^\frac{1}{\sigma}}
& \geq  1+2(\kappa_0-\kappa)\left< W \right>^{\frac{1}{\sigma}} \\
& \gtrsim  (\kappa_0-\kappa)^r \left< W \right>^{\frac{r}{\sigma}}
\end{aligned}
\end{equation}
where $0 \leq r \leq 1$. Since $\frac{1}{\sigma}> \max(0,(2A-1))$,
 we can always find an $r \in [0,1)$ such that $\frac{r}{\sigma} \geq 
\max(0,(2A-1))$. For any such
value of $r$, we have:
\begin{equation}
I_3^\prime \lesssim (\kappa_0-\kappa)^{-r}
\int_{\mathbb{R}^d} dz \frac{\left< W \right>}
{\left<W-z\right>\left<z\right>^{2\beta}}
\end{equation}
Splitting the integral into the regions $|z| < \frac{1}{2} |W|$,
$|z|>2|W|$, and $\frac{1}{2} |W| \leq |z| \leq 2|W|$, we are able
to show that $I_3 \lesssim (\kappa_0 - \kappa)^{-r}$ uniformly
in $W$ as long as $\beta > \frac{d+1}{2}$,
$\frac{1}{\sigma} > \max(0,(2A-1))$, and $r\in [0,1)$ is such that
$\frac{r}{\sigma} \geq \max(0,(2A-1))$. 

Let us turn to $I_2$; clearly, the integral over the set
$|z-W|<1$ is uniformly bounded in $W$ if $\beta \geq A$. Therefore,
it suffices to bound the following integral uniformly in $W$:
\begin{equation}
I_2^\prime = \int_{\mathbb{R}^d} dz
\frac{1}{\left<W-z\right> \left< z \right>^{2\beta - 2A}}
\end{equation}
For any $A\in [0,1]$, this integral is automatically bounded,
uniformly in $W$, if $\beta > \frac{d+2}{2}$.

Finally we turn to $I_1$:
\begin{equation}
I_1 = \sup_{P\subset \mathbb{R}^d : \dim P = d-1}
\int_P dS(w) \frac{\left<W\right>^{2\alpha}}
{\left< W - w\right>^{2\alpha} \left< w \right>^{2\alpha}}
\end{equation}
We consider separately the regions $|w|<\frac{1}{2} |W|$,
$|w|>2|W|$, and $\frac{1}{2} |W| \leq |w| \leq 2|W|$; we find that
the integral is uniformly bounded in $W$ and $P$ as long
as $\alpha > \frac{d-1}{2}$.

To summarize, as long as $\alpha > \frac{d-1}{2}$, $\beta > d$,
$\frac{1}{\sigma} > \max (0,(2A-1))$,
and $r\in [0,1)$ is chosen such that
 $\frac{r}{\sigma}\geq \max(0,(2A-1))$, then 
for all $\kappa_0 > \kappa > 0$:
\begin{equation}
\begin{aligned}
& \left\Vert B^-_{i,k+1} \left[ 
e^{\frac{1}{2} i t \left( \Delta_{X_{k+1}} - \Delta_{X_{k+1}^\prime}\right)}
\gamma_0^{(k+1)} \right] \right\Vert_{L^2_t 
H^{\alpha,\beta,\sigma,\kappa}_k } \leq \\
& \qquad \qquad \qquad \leq
C(\alpha,\beta,\sigma,r) \left\Vert \mathbf{b} \right\Vert_{L^\infty_A}
\left( 1+\left(\kappa_0-\kappa\right)^{-\frac{1}{2}r}
\right)
 \left\Vert \gamma_0^{(k+1)} 
\right\Vert_{H^{\alpha,\beta,\sigma,\kappa_0}_{k+1}}
\end{aligned}
\end{equation}

\subsection*{Gain Term}

Consider a typical part of the gain term, e.g.
$B_{1,k+1}^+ \gamma^{(k+1)}$:
\begin{equation}
\begin{aligned}
& B_{1,k+1}^+ \gamma^{(k+1)} (t,X_k,X_k^\prime) = \\
& =  \frac{i}{2^{2d} \pi^d} \int_{\mathbb{S}^{d-1}}
d\omega \int_{\mathbb{R}^d} dz 
\hat{\mathbf{b}}^\omega \left( \frac{z}{2}\right)  \times \\
& \times \gamma^{(k+1)} \left( t,
x_1-\frac{1}{2}P_\omega (x_1-x_1^\prime)-
\frac{R_\omega (z)}{4} ,X_{2:k}, \right. \\
& \qquad \qquad \qquad \qquad\qquad
\left. \frac{x_1+x_1^\prime}{2}+\frac{1}{2} P_\omega
(x_1-x_1^\prime)+\frac{R_\omega (z)}{4},\right.\\
& \qquad \qquad \left.
 x_1^\prime + \frac{1}{2} P_\omega (x_1-x_1^\prime)
+\frac{R_\omega (z)}{4}, X_{2:k}^\prime, \right. \\
& \qquad \qquad \qquad \qquad\qquad \left.
\frac{x_1+x_1^\prime}{2}-\frac{1}{2} P_\omega (x_1-x_1^\prime)
-\frac{R_\omega (z)}{4}\right) 
\end{aligned}
\end{equation}
The spacetime Fourier transform of the function
\begin{equation}
B^+_{1,k+1}
\left[ e^{\frac{1}{2} i t \left( \Delta_{X_{k+1}}- 
\Delta_{X_{k+1}^\prime}\right)}
\gamma_0^{(k+1)} \right] (t,X_k,X_k^\prime)
\end{equation}
is the following, up to a constant:
\begin{equation}
\begin{aligned}
& \int_{\mathbb{S}^{d-1}} d\omega 
\int d\eta_1 d\eta_1^\prime d\eta_2 d\eta_2^\prime\times \\
& \times \delta \left( \tau + \frac{1}{2} |\eta_1|^2 - \frac{1}{2} 
|\eta_1^\prime|^2 + \frac{1}{2} |\eta_2|^2 -
\frac{1}{2} |\eta_2^\prime|^2 +\frac{1}{2}\sum_{2\leq i \leq k}
(|\xi_i|^2-|\xi_i^\prime|^2)\right) \times \\
& \times \delta \left( -\xi_1 + \eta_1 
+ \frac{\eta_2+\eta_2^\prime}{2}
-\frac{1}{2} P_\omega (\eta_1-\eta_1^\prime) 
+\frac{1}{2} P_\omega (\eta_2-\eta_2^\prime)\right)\times\\
& \times \delta \left( -\xi_1^\prime + \eta_1^\prime +
\frac{\eta_2+\eta_2^\prime}{2}
+ \frac{1}{2} P_\omega (\eta_1-\eta_1^\prime) 
- \frac{1}{2} P_\omega \cdot (\eta_2 - \eta_2^\prime)\right) \times \\
& \times \mathbf{b}
 \left( \frac{|-\eta_1+\eta_1^\prime+\eta_2-\eta_2^\prime|}{2},
\omega\cdot\frac{R_\omega ( -\eta_1+\eta_1^\prime+\eta_2-\eta_2^\prime)}
{|-\eta_1+\eta_1^\prime+\eta_2-\eta_2^\prime|}\right) \times \\
& \times \hat{\gamma}_0^{(k+1)} 
\left( \eta_1,\xi_2,\dots,\xi_k,\eta_2,
\eta_1^\prime,\xi_2^\prime,\dots,\xi_k^\prime,\eta_2^\prime\right)
\end{aligned}
\end{equation}
This is bounded by $\left\Vert \mathbf{b}
 \right\Vert_{L^\infty_A}$ times the
following integral:
\begin{equation}
\label{eq:spacetime-fourier}
\begin{aligned}
& \int_{\mathbb{S}^{d-1}} d\omega 
\int d\eta_1 d\eta_1^\prime d\eta_2 d\eta_2^\prime
\left( \left<\eta_1-\eta_1^\prime\right>^A +
\left<\eta_2-\eta_2^\prime\right>^A \right)\times \\
& \times \delta \left( \tau + \frac{1}{2} |\eta_1|^2 - \frac{1}{2} 
|\eta_1^\prime|^2 + \frac{1}{2} |\eta_2|^2 -
\frac{1}{2} |\eta_2^\prime|^2 +\frac{1}{2}\sum_{2\leq i \leq k}
(|\xi_i|^2-|\xi_i^\prime|^2)\right) \times \\
& \times \delta \left( -\xi_1 + \eta_1 
+ \frac{\eta_2+\eta_2^\prime}{2}
-\frac{1}{2} P_\omega (\eta_1-\eta_1^\prime) 
+\frac{1}{2} P_\omega (\eta_2-\eta_2^\prime)\right)\times\\
& \times \delta \left( -\xi_1^\prime + \eta_1^\prime +
\frac{\eta_2+\eta_2^\prime}{2}
+ \frac{1}{2} P_\omega (\eta_1-\eta_1^\prime) 
- \frac{1}{2} P_\omega (\eta_2 - \eta_2^\prime)\right) \times \\
& \times \left| \hat{\gamma}_0^{(k+1)} 
\left( \eta_1,\xi_2,\dots,\xi_k,\eta_2,
\eta_1^\prime,\xi_2^\prime,\dots,\xi_k^\prime,\eta_2^\prime\right)\right|
\end{aligned}
\end{equation}
Introduce the change of variables $w_1 = \frac{\eta_1+\eta_1^\prime}{2}$,
$z_1 = \frac{\eta_1 - \eta_1^\prime}{2}$, 
$w_2 = \frac{\eta_2 + \eta_2^\prime}{2}$,
$z_2 = \frac{\eta_2-\eta_2^\prime}{2}$.
Then (\ref{eq:spacetime-fourier}) becomes
\begin{equation}
\label{eq:spacetime-fourier-2}
\begin{aligned}
& \int_{\mathbb{S}^{d-1}} d\omega \int dw_1 dz_1 dw_2 dz_2
\left( \left< 2z_1 \right>^A + \left< 2z_2 \right>^A \right) \times \\
&  \times \delta \left( \tau + \frac{1}{2} |w_1+z_1|^2 - \frac{1}{2} 
|w_1-z_1|^2 + \frac{1}{2} |w_2+z_2|^2 -
\frac{1}{2} |w_2-z_2|^2 + \right. \\
& \qquad \qquad \qquad \qquad \qquad\left. + \frac{1}{2}\sum_{2\leq i \leq k} 
(|\xi_i|^2 - |\xi_i^\prime|^2)\right) \times \\
& \times \delta \left( -\xi_1 + w_1+z_1+ w_2
-  P_\omega (z_1-z_2) \right)\times\\
& \times \delta \left( -\xi_1^\prime + w_1 - z_1 + w_2
+ P_\omega (z_1 -z_2) \right) \times \\
& \times \left| \hat{\gamma}_0^{(k+1)} 
\left( w_1+z_1,\xi_2,\dots,\xi_k,w_2+z_2,
w_1-z_1,\xi_2^\prime,\dots,\xi_k^\prime,w_2-z_2\right)\right|
\end{aligned}
\end{equation}
Introduce yet another change of variables
$r_1 = \frac{w_1+w_2}{2}$, $s_1 = \frac{w_1-w_2}{2}$,
$r_2 = \frac{z_1+z_2}{2}$, $s_2 = \frac{z_1-z_2}{2}$.
Then (\ref{eq:spacetime-fourier-2}) becomes
\begin{equation}
\begin{aligned}
& \int_{\mathbb{S}^{d-1}} d\omega \int dr_1 ds_1 dr_2 ds_2
\left( \left<2(r_2+s_2)\right>^A + 
\left<2(r_2-s_2)\right>^A\right) \times \\
&  \times \delta \left( \tau + \frac{1}{2} |r_1+s_1+r_2+s_2|^2 - \frac{1}{2} 
|r_1+s_1-r_2-s_2|^2 + \right. \\
& \left. \qquad + \frac{1}{2} |r_1-s_1+r_2-s_2|^2 -
\frac{1}{2} |r_1-s_1-r_2+s_2|^2  + \right. \\
& \qquad \qquad \qquad \qquad \qquad \left. +\frac{1}{2}\sum_{2\leq i \leq k}
(|\xi_i|^2-|\xi_i^\prime|^2)\right) \times \\
& \times \delta \left( -\xi_1 + 2r_1 + r_2 + R_\omega (s_2) \right)\times\\
& \times \delta \left( -\xi_1^\prime + 2r_1 - r_2 - R_\omega (s_2)
 \right) \times \\
& \times\left| \hat{\gamma}_0^{(k+1)} 
\left( r_1+s_1+r_2+s_2,\xi_2,\dots,\xi_k,r_1-s_1+r_2-s_2,\right.\right.\\
& \qquad \qquad \qquad\left.\left. r_1+s_1-r_2-s_2,
\xi_2^\prime,\dots,\xi_k^\prime,r_1-s_1-r_2+s_2\right)\right|
\end{aligned}
\end{equation}
Replace $r_1$ with $\frac{r_1}{2}$ throughout:
\begin{equation}
\begin{aligned}
& \int_{\mathbb{S}^{d-1}} d\omega \int dr_1 ds_1 dr_2 ds_2 
\left( \left<2(r_2+s_2)\right>^A + 
\left<2(r_2-s_2)\right>^A\right) \times \\
&  \times \delta \left( \tau + \frac{1}{2} 
\left|\frac{r_1}{2}+s_1+r_2+s_2\right|^2 - 
\frac{1}{2} 
\left|\frac{r_1}{2}+s_1-r_2-s_2\right|^2 + \right. \\
& \left. \qquad + \frac{1}{2} \left|\frac{r_1}{2}-s_1+r_2-s_2\right|^2 -
\frac{1}{2} \left|\frac{r_1}{2}-s_1-r_2+s_2\right|^2 + \right. \\
& \qquad \qquad \qquad \qquad \qquad \left. +\frac{1}{2}
\sum_{2\leq i \leq k} (|\xi_i|^2-|\xi_i^\prime|^2)\right) \times \\
& \times \delta \left( -\xi_1 + r_1 + r_2 + R_\omega (s_2) \right)\times\\
& \times \delta \left( -\xi_1^\prime + r_1 - r_2 - R_\omega (s_2)
  \right) \times \\
& \times \left| \hat{\gamma}_0^{(k+1)} 
\left( \frac{r_1}{2}+s_1+r_2+s_2,\xi_2,\dots,\xi_k,
\frac{r_1}{2}-s_1+r_2-s_2,\right.\right.\\
& \qquad \qquad \qquad\left.\left. \frac{r_1}{2}+s_1-r_2-s_2,
\xi_2^\prime,\dots,\xi_k^\prime,\frac{r_1}{2}-s_1-r_2+s_2\right)\right|
\end{aligned}
\end{equation}
Finally perform the change of variables
$\zeta_1 = r_1+r_2$, $\zeta_2 = r_1-r_2$:
\begin{equation}
\begin{aligned}
& \int_{\mathbb{S}^{d-1}} d\omega \int d\zeta_1 d\zeta_2 ds_1 ds_2 
\left( \left< \zeta_1-\zeta_2+2s_2\right>^A + 
\left< \zeta_1-\zeta_2-2s_2\right>^A \right) \times \\
&  \times \delta \left( \tau + \frac{1}{2} 
\left|\frac{3\zeta_1}{4}-\frac{\zeta_2}{4}+s_1+s_2\right|^2 - 
\frac{1}{2} 
\left|-\frac{\zeta_1}{4}+\frac{3\zeta_2}{4}+s_1-s_2\right|^2 + \right. \\
& \left. \qquad + \frac{1}{2} 
\left|\frac{3\zeta_1}{4}-\frac{\zeta_2}{4}-s_1-s_2\right|^2 -
\frac{1}{2} \left|-\frac{\zeta_1}{4}+\frac{3\zeta_2}{4}
-s_1+s_2\right|^2 + \right. \\
& \qquad \qquad \qquad \qquad \qquad \left. +\frac{1}{2}
\sum_{2\leq i \leq k} (|\xi_i|^2-|\xi_i^\prime|^2)\right) \times \\
& \times \delta \left( -\xi_1 + \zeta_1 + R_\omega (s_2) \right)\times\\
& \times \delta \left( -\xi_1^\prime + \zeta_2 - R_\omega (s_2)
 \right) \times \\
& \times \left| \hat{\gamma}_0^{(k+1)} 
\left( \frac{3\zeta_1}{4}-\frac{\zeta_2}{4}+s_1+s_2,\xi_2,\dots,\xi_k,
\frac{3\zeta_1}{4}-\frac{\zeta_2}{4}-s_1-s_2,\right.\right.\\
& \qquad \qquad \qquad\left.\left. 
\frac{-\zeta_1}{4}+\frac{3\zeta_2}{4}+s_1-s_2,
\xi_2^\prime,\dots,\xi_k^\prime,
-\frac{\zeta_1}{4}+\frac{3\zeta_2}{4}-s_1+s_2\right)\right|
\end{aligned}
\end{equation}
Now we can integrate out the variables $\zeta_1,\zeta_2$ to obtain:
\begin{equation}
\begin{aligned}
& \int_{\mathbb{S}^{d-1}} d\omega \int ds_1 ds_2 
\left( \left< 4 s_2^\Vert + \xi_1 - \xi_1^\prime \right>^A +
\left< -4s_2^\bot + \xi_1 - \xi_1^\prime \right>^A \right)\times \\
&  \times \delta \left( \tau + \frac{1}{2} 
\left|s_1 + 2  s_2^{\Vert} + \frac{3\xi_1-\xi_1^\prime}{4}
\right|^2 - 
\frac{1}{2} 
\left|s_1-2 s_2^{\Vert} + \frac{3\xi_1^\prime-\xi_1}{4}
\right|^2 + \right. \\
& \left. \qquad \qquad + \frac{1}{2} 
\left|-s_1 - 2 s_2^{\bot} +
\frac{3\xi_1-\xi_1^\prime}{4}\right|^2 -
\frac{1}{2} \left|
-s_1 + 2 s_2^{\bot} + 
\frac{3\xi_1^\prime-\xi_1}{4} \right|^2 + \right. \\
& \qquad \qquad \qquad \qquad \qquad \left. + \frac{1}{2}\sum_{2\leq i \leq k}
(|\xi_i|^2 - |\xi_i^\prime|^2) \right) \times \\
& \times \left| \hat{\gamma}_0^{(k+1)} 
\left( s_1+2s_2^\Vert+\frac{3\xi_1-\xi_1^\prime}{4},\xi_2,\dots,\xi_k,
-s_1-2 s_2^\bot + \frac{3\xi_1-\xi_1^\prime}{4},\right.\right.\\
& \qquad \qquad \qquad\left. \left.
s_1-2s_2^\Vert + \frac{3\xi_1^\prime-\xi_1}{4},
\xi_2^\prime,\dots,\xi_k^\prime,
-s_1+2s_2^\bot+\frac{3\xi_1^\prime-\xi_1}{4}\right)\right|
\end{aligned}
\end{equation}
where $s_2^{\Vert}=P_\omega ( s_2)$ and
$s_2^{\bot} = \left(\mathbb{I}-P_\omega\right) (s_2)$.

We want to estimate the following integral, for suitable 
$\alpha,\beta,\kappa,\sigma > 0$:
\begin{equation}
\begin{aligned}
I^+ (\alpha,\beta,\kappa,\sigma) & = \int d\tau d\xi_1\dots d\xi_k
d\xi_1^\prime \dots d\xi_k^\prime \times \\
& \qquad \qquad \qquad \times \prod_{i=1}^k
\left\{ \left< \xi_i + \xi_i^\prime \right>^{2\alpha}
\left< \xi_i - \xi_i^\prime \right>^{2\beta}
e^{2\kappa \left< \xi_i - \xi_i^\prime\right>^{\frac{1}{\sigma}}}
\right\}\times \\
& \qquad \qquad \qquad \times 
\left|\left(B_{1,k+1}^+
\left[ e^{\frac{1}{2} i t \left( \Delta_{X_{k+1}}- 
\Delta_{X_{k+1}}^\prime\right)}
\gamma_0^{(k+1)} \right] \right)^\sim \right|^2
\end{aligned}
\end{equation}
Reasoning as for the loss term, if we can show that the following integral
\begin{equation}
\label{eq:gain-int}
\begin{aligned}
& \int_{\mathbb{S}^{d-1}} d\omega \int ds_1 ds_2 \times \\
&  \times \delta \left( \tau + \frac{1}{2} 
\left|s_1 + 2  s_2^{\Vert} + \frac{3\xi_1-\xi_1^\prime}{4}
\right|^2 - 
\frac{1}{2} 
\left|s_1-2 s_2^{\Vert} + \frac{3\xi_1^\prime-\xi_1}{4}
\right|^2 + \right. \\
& \left. \qquad \qquad + \frac{1}{2} 
\left|-s_1 - 2 s_2^{\bot} +
\frac{3\xi_1-\xi_1^\prime}{4}\right|^2 -
\frac{1}{2} \left|
-s_1 + 2 s_2^{\bot} + 
\frac{3\xi_1^\prime-\xi_1}{4} \right|^2 + \right. \\
& \qquad \qquad\qquad \qquad \qquad\qquad
\left. + \sum_{2\leq i \leq k} (|\xi_i|^2-|\xi_i^\prime|^2) \right) \times \\
& \times \frac{\left<\xi_1+\xi_1^\prime\right>^{2\alpha} 
\left<\xi_1-\xi_1^\prime\right>^{2\beta}
\left( \left< 4 s_2^\Vert + \xi_1 - \xi_1^\prime\right>^{2A} +
\left< -4s_2^\bot + \xi_1 - \xi_1^\prime\right>^{2A}\right)}
{\left<2s_1+\frac{\xi_1+\xi_1^\prime}{2}\right>^{2\alpha}
\left<4s_2^{\Vert} + \xi_1 - \xi_1^\prime\right>^{2\beta}
\left< -2s_1 + \frac{\xi_1+\xi_1^\prime}{2}\right>^{2\alpha}
\left< -4s_2^{\bot}+\xi_1-\xi_1^\prime\right>^{2\beta}}\times \\
& \times
e^{-2(\kappa_0-\kappa)
\left<4s_2^\Vert+\xi_1-\xi_1^\prime\right>^{\frac{1}{\sigma}}}
e^{-2(\kappa_0-\kappa)
\left<-4s_2^\bot+\xi_1-\xi_1^\prime\right>^{\frac{1}{\sigma}}}
\times \\
& \times 
e^{2\kappa \left( \left< \xi_1-\xi_1^\prime\right>^{\frac{1}{\sigma}}
-\left<4s_2^\Vert+\xi_1-\xi_1^\prime\right>^{\frac{1}{\sigma}}
-\left<-4s_2^\bot+\xi_1-\xi_1^\prime\right>^{\frac{1}{\sigma}}\right)}
\end{aligned}
\end{equation}
is bounded uniformly in $\tau,\xi_1,\dots,\xi_k,
\xi_1^\prime,\dots,\xi_k^\prime$, then we will have the
following estimate:
\begin{equation}
\begin{aligned}
& \left\Vert B^+_{1,k+1} \left[ 
e^{\frac{1}{2} i t \left( \Delta_{X_{k+1}} - \Delta_{X_{k+1}^\prime}\right)}
\gamma_0^{(k+1)} \right] \right\Vert_{L^2_t 
H^{\alpha,\beta,\sigma,\kappa}_k } \leq \\
& \qquad \qquad \qquad \qquad  \leq
C(\alpha,\beta,\sigma,\kappa,\kappa_0) \left\Vert \mathbf{b}
 \right\Vert_{L^\infty_A}
 \left\Vert \gamma_0^{(k+1)}
 \right\Vert_{H^{\alpha,\beta,\sigma,\kappa_0}_{k+1}}
\end{aligned}
\end{equation}

Before proceeding further, we must eliminate the most dangerous
contribution in (\ref{eq:gain-int}), which is the following
\emph{exponential} factor:
\begin{equation}
e^{2\kappa \left( \left< \xi_1-\xi_1^\prime\right>^{\frac{1}{\sigma}}
-\left<4s_2^\Vert+\xi_1-\xi_1^\prime\right>^{\frac{1}{\sigma}}
-\left<-4s_2^\bot+\xi_1-\xi_1^\prime\right>^{\frac{1}{\sigma}}\right)}
\end{equation}
We will show that this factor is in fact bounded by $1$, as long
as $\sigma \geq \frac{1}{2}$. Indeed for $\sigma \geq \frac{1}{2}$ we have:
\begin{equation}
\begin{aligned}
& \left< \xi_1-\xi_1^\prime\right>^{\frac{1}{\sigma}}
-\left<4s_2^\Vert+\xi_1-\xi_1^\prime\right>^{\frac{1}{\sigma}}
-\left<-4s_2^\bot+\xi_1-\xi_1^\prime\right>^{\frac{1}{\sigma}} \\
& \leq \left< \xi_1-\xi_1^\prime\right>^{\frac{1}{\sigma}}
-\left<(\xi_1-\xi_1^\prime)^\bot\right>^{\frac{1}{\sigma}}
-\left<(\xi_1-\xi_1^\prime)^\Vert\right>^{\frac{1}{\sigma}} \\
&  \leq \left( \left< (\xi_1-\xi_1^\prime)^\Vert\right>^2
+ \left< (\xi_1-\xi_1^\prime)^\bot\right>^2
\right)^{\frac{1}{2\sigma}}
-\left<(\xi_1-\xi_1^\prime)^\bot\right>^{\frac{1}{\sigma}}
-\left<(\xi_1-\xi_1^\prime)^\Vert\right>^{\frac{1}{\sigma}} \\
&  \leq 0 
\end{aligned}
\end{equation} 

We now deal with the other exponential factors in (\ref{eq:gain-int}),
namely:
\begin{equation}
\label{eq:gain-exp-11}
e^{-2(\kappa_0-\kappa)\left<4s_2^\Vert+\xi_1-\xi_1^\prime
\right>^{\frac{1}{\sigma}}}
e^{-2(\kappa_0-\kappa)\left<-4s_2^\bot+\xi_1-\xi_1^\prime
\right>^{\frac{1}{\sigma}}}
\end{equation}
Since $\frac{1}{\sigma} > \max (0,2A-1)$, we can always find
$r \in [0,1)$ such that $\frac{r}{\sigma} \geq 
\max (0,2A-1+\delta)$ for a small $\delta > 0$. Since
$e^u \geq 1 + u \gtrsim u^r$ for $u>0$, we find that if $\kappa <
\kappa_0$ then  (\ref{eq:gain-exp-11}) is bounded above by the following
quantity:
\begin{equation}
\begin{aligned}
& (\kappa_0 - \kappa)^{-r}\times \\
& \times \min
\left( \left< 4 s_2^\Vert + \xi_1 - \xi_1^\prime\right>^{-\max(0,2A-1+\delta)},
\left< -4 s_2^\bot + \xi_1 - \xi_1^\prime\right>^{-\max(0,2A-1+\delta)}
\right)
\end{aligned}
\end{equation}

The integral (\ref{eq:gain-int}) is now bounded by the following integral,
if $\delta > 0$ is sufficiently small depending on $A$, $\sigma$ and $r$:
(note that this follows from the previous paragraph by considering
separately $0\leq A < \frac{1}{2}$ and $\frac{1}{2} \leq A \leq 1$)
\begin{equation}
\begin{aligned}
& (\kappa_0 - \kappa)^{-r}
 \int_{\mathbb{S}^{d-1}} d\omega \int ds_1 ds_2\times \\
& \times \delta \left( \tau + \frac{1}{2} \sum_{i=1}^k
\left( |\xi_i|^2 - |\xi_i^\prime|^2
\right) +
\left( 4 s_1 - R_\omega ( \xi_1+\xi_1^\prime)\right)\cdot s_2 \right)\times\\
& \times  \frac{\left<\xi_1+\xi_1^\prime\right>^{2\alpha} 
\left<\xi_1-\xi_1^\prime\right>^{2\beta}\left(
\left< 4s_2^\Vert + \xi_1 -\xi_1^\prime\right>^{1-\delta} +
\left< -4s_2^\bot + \xi_1 - \xi_1^\prime\right>^{1-\delta}\right)}
{\left<2s_1+\frac{\xi_1+\xi_1^\prime}{2}\right>^{2\alpha}
\left<4s_2^{\Vert} + \xi_1 - \xi_1^\prime\right>^{2\beta}
\left< -2s_1 + \frac{\xi_1+\xi_1^\prime}{2}\right>^{2\alpha}
\left< -4s_2^{\bot}+\xi_1-\xi_1^\prime\right>^{2\beta}}
\end{aligned}
\end{equation}
 This is in turn equivalent to the
following integral:
\begin{equation}
\begin{aligned}
& (\kappa_0 - \kappa)^{-r} \times \\
&\times \int_{\mathbb{S}^{d-1}} d\omega \int ds_2
\frac{\left<\xi_1-\xi_1^\prime\right>^{2\beta}
\left( \left< 4 s_2^\Vert + \xi_1 - \xi_1^\prime \right>^{1-\delta}
+ \left< - 4 s_2^\bot + \xi_1 - \xi_1^\prime \right>^{1-\delta}\right)}
{\left| 4 s_2\right|
\left< 4s_2^{\Vert}+\xi_1-\xi_1^\prime  \right>^{2\beta}
\left< -4s_2^{\bot}+\xi_1-\xi_1^\prime  \right>^{2\beta}}\times \\
& \qquad \qquad \qquad \qquad\qquad
\times \int_P dS(s_1) \frac{\left<\xi_1+\xi_1^\prime\right>^{2\alpha}}
{\left< 2s_1 + \frac{\xi_1+\xi_1^\prime}{2} \right>^{2\alpha}
\left< -2s_1 + \frac{\xi_1+\xi_1^\prime}{2}  \right>^{2\alpha}}
\end{aligned}
\end{equation}
where $P \subset \mathbb{R}^d$ is the following codimension one hyperplane:
\begin{equation}
P = \left\{ s_1 \in \mathbb{R}^d \left|
\tau + \frac{1}{2} \sum_{i=1}^k \left( |\xi_i|^2-|\xi_i^\prime|^2\right) +
\left( 4s_1-R_\omega (\xi_1+\xi_1^\prime)\right)\cdot s_2 = 0\right.\right\}
\end{equation}
Therefore we only need to show the boundedness of the following three
quantities uniformly in $\xi_1,\xi_1^\prime,\tau$:
\begin{equation}
I_1 = \sup_{P \subset \mathbb{R}^d : \dim P = d-1}
\int_P dS(s) \frac{\left<\xi_1+\xi_1^\prime\right>^{2\alpha}}
{\left<2s+\frac{\xi_1+\xi_1^\prime}{2}\right>^{2\alpha}
\left<-2s+\frac{\xi_1+\xi_1^\prime}{2}\right>^{2\alpha}}
\end{equation}
\begin{equation}
I_2 = \int_{\mathbb{S}^{d-1}} d\omega \int_{\mathbb{R}^d} ds
\frac{\left<\xi_1-\xi_1^\prime\right>^{2\beta}}
{|4s|\left<4s^{\Vert}+\xi_1-\xi_1^\prime\right>^{2\beta-1+\delta}
\left<-4s^{\bot}+\xi_1-\xi_1^\prime\right>^{2\beta}}
\end{equation}
\begin{equation}
I_3 = \int_{\mathbb{S}^{d-1}} d\omega \int_{\mathbb{R}^d} ds
\frac{\left<\xi_1-\xi_1^\prime\right>^{2\beta}}
{|4s|\left<4s^{\Vert}+\xi_1-\xi_1^\prime\right>^{2\beta}
\left<-4s^{\bot}+\xi_1-\xi_1^\prime\right>^{2\beta-1+\delta}}
\end{equation}

Let us first consider the integral $I_2$; in what follows we will
assume that $\beta > \frac{d}{2}$. Clearly,
 $I_2$ is equivalent to the following quantity:
\begin{equation}
I_2 \lesssim \int_{\mathbb{S}^{d-1}} d\omega \int_{\mathbb{R}^d} ds
\frac{\left<\xi_1-\xi_1^\prime\right>^{2\beta}}
{|s|\left<s^{\Vert}+\xi_1-\xi_1^\prime\right>^{2\beta-1+\delta}
\left<s^{\bot}+ \xi_1-\xi_1^\prime \right>^{2\beta}}
\end{equation}
Setting $W = \xi_1 - \xi_1^\prime$, this gives:
\begin{equation}
I_2 \lesssim \int_{\mathbb{S}^{d-1}} d\omega \int_{\mathbb{R}^d} ds
\frac{\left<W\right>^{2\beta}}
{|s|\left<s^{\Vert}+W\right>^{2\beta-1+\delta}
\left<s^{\bot}+W\right>^{2\beta}}
\end{equation}
Moreover, since the integral for $|s|\leq 1$ is obviously uniformly
bounded in $W$, we may instead bound the following integral:
\begin{equation}
I_2^\prime \lesssim \int_{\mathbb{S}^{d-1}} d\omega \int_{\mathbb{R}^d} ds
\frac{\left<W\right>^{2\beta}}
{\left<s\right>\left<s^{\Vert}+W\right>^{2\beta-1+\delta}
\left<s^{\bot}+W\right>^{2\beta}}
\end{equation}
Since $|s^\Vert| \leq |s|$ we have:
\begin{equation}
I_2^\prime \lesssim \int_{\mathbb{S}^{d-1}} d\omega \int_{\mathbb{R}^d} ds
\frac{\left<W\right>^{2\beta}}
{\left<s^\Vert\right>\left<s^{\Vert}+W\right>^{2\beta-1+\delta}
\left<s^{\bot}+W\right>^{2\beta}}
\end{equation}
Therefore, for all large enough $|W|$,
\begin{equation}
\begin{aligned}
I_2^\prime & \lesssim \int_{\mathbb{S}^{d-1}} d\omega \int_{\mathbb{R}^d} ds
\frac{\left<W\right>^{2\beta}}
{\left< s^\Vert \right>
\left<s^{\Vert}+W\right>^{2\beta-1+\delta}
\left<s^{\bot}+W\right>^{2\beta}} \\
& = \int_{\mathbb{S}^{d-1}} d\omega \left< W \right>^{2\beta}
\left( \int \frac{ds^\Vert}
{\left< s^\Vert \right>
\left< s^\Vert + W \right>^{2\beta-1+\delta}}\right)
\left( \int \frac{ds^\bot}
{\left< s^\bot + W \right>^{2\beta}}\right) \\
& \lesssim \int_{\mathbb{S}^{d-1}} d\omega
\left< W \right>^{2\beta} \left(
\left< W\right>^{-1} \left<W^\bot\right>^{2-2\beta-\delta}
\log \left<W\right> \right)
\left( \left<W^\Vert\right>^{d-1-2\beta}\right)\\
\end{aligned}
\end{equation}
The integral over $s^\bot$ is estimated by a trivial
computation, whereas the integral over $s^\Vert$ may be estimated
by considering separately the regions $|s^\Vert| < \frac{1}{2} |W|$,
$|s^\Vert| > 2|W|$, and $\frac{1}{2} |W| \leq |s^\Vert| \leq 2|W|$.

We find that $I_2^\prime$ obeys the following
estimate:
\begin{equation}
I_2^\prime \lesssim
 \int_{\mathbb{S}^{d-1}} d\omega
\left< W \right>^{2\beta-1+\frac{1}{2}\delta}
 \left<W^\bot\right>^{2-2\beta-\delta}
\left<W^\Vert\right>^{d-1-2\beta}
\end{equation}
Then we have
\begin{equation}
\left< W \right>^{2\beta-1+\frac{1}{2}\delta} \lesssim
\left< W^\Vert \right>^{2\beta-1+\frac{1}{2}\delta} +
\left< W^\bot \right>^{2\beta-1+\frac{1}{2}\delta}
\end{equation}
Hence $I_2^\prime \lesssim I_2^{\prime \prime} + 
I_2^{\prime \prime \prime}$ where
\begin{equation}
I_2^{\prime \prime} = \int_{\mathbb{S}^{d-1}} d\omega
\left< W^\bot\right>^{2-2\beta-\delta}
\left< W^\Vert\right>^{d-2+\frac{1}{2}\delta}
\end{equation}
\begin{equation}
I_2^{\prime \prime \prime} = \int_{\mathbb{S}^{d-1}}
d\omega \left< W^\bot\right>^{1-\frac{1}{2}\delta}
\left< W^\Vert \right>^{d-1-2\beta}
\end{equation}
Then for any $\delta$ sufficiently small and
 $\beta$ sufficiently large ($\beta \geq d$ is easily
sufficient for small $\delta$), both $I_2^{\prime \prime}$
and $I_2^{\prime \prime \prime}$ may be bounded using dyadic
decompositions in the angular parameter $\omega$, as follows:
neglecting additive constants,
\begin{equation}
\begin{aligned}
I_2^{\prime \prime} & \lesssim \sum_{k=1}^\infty
\int_{\omega : 2^{-k-1} |W^\Vert| \leq |W^\bot| < 2^{-k} |W^\Vert|}
d\omega \left< W^\bot\right>^{2-2\beta-\delta}
\left< W^\Vert\right>^{d-2+\frac{1}{2}\delta} \\
& \lesssim \sum_{k=1}^\infty 2^{-k-1} \times
(2^{-k})^{d-2} \times (2^{k+1})^{d-2+\frac{1}{2}\delta} <\infty
\end{aligned}
\end{equation}
\begin{equation}
\begin{aligned}
I_2^{\prime \prime \prime} & \lesssim \sum_{k=1}^\infty
\int_{\omega : 2^{-k-1} |W^\bot| \leq |W^\Vert| < 2^{-k} |W^\bot|}
d\omega \left< W^\bot\right>^{1-\frac{1}{2}\delta}
\left< W^\Vert\right>^{d-1-2\beta} \\
& \lesssim \sum_{k=1}^\infty 2^{-k-1} \times 
(2^{k+1})^{1-\frac{1}{2}\delta} <\infty
\end{aligned}
\end{equation}
The factor of $(2^{-k})^{d-2}$ in $I_2^{\prime \prime}$ comes from
the Jacobian for spherical coordinates in $\mathbb{R}^d$.

Let us now consider the integral $I_3$, and assume
$\beta > \frac{d}{2}$. Clearly,
 $I_3$ is equivalent to the following quantity:
\begin{equation}
I_3 \lesssim \int_{\mathbb{S}^{d-1}} d\omega \int_{\mathbb{R}^d} ds
\frac{\left<\xi_1-\xi_1^\prime\right>^{2\beta}}
{|s|\left<s^{\Vert}+\xi_1-\xi_1^\prime\right>^{2\beta}
\left<s^{\bot}+ \xi_1-\xi_1^\prime \right>^{2\beta-1+\delta}}
\end{equation}
Setting $W = \xi_1 - \xi_1^\prime$, this gives:
\begin{equation}
I_3 \lesssim \int_{\mathbb{S}^{d-1}} d\omega \int_{\mathbb{R}^d} ds
\frac{\left<W\right>^{2\beta}}
{|s|\left<s^{\Vert}+W\right>^{2\beta}
\left<s^{\bot}+W\right>^{2\beta-1+\delta}}
\end{equation}
Moreover, since the integral for $|s|\leq 1$ is obviously uniformly
bounded in $W$, we may instead bound the following integral:
\begin{equation}
I_3^\prime \lesssim \int_{\mathbb{S}^{d-1}} d\omega \int_{\mathbb{R}^d} ds
\frac{\left<W\right>^{2\beta}}
{\left<s\right>\left<s^{\Vert}+W\right>^{2\beta}
\left<s^{\bot}+W\right>^{2\beta-1+\delta}}
\end{equation}
Since $|s^\Vert| \leq |s|$ we have:
\begin{equation}
I_3^\prime \lesssim \int_{\mathbb{S}^{d-1}} d\omega \int_{\mathbb{R}^d} ds
\frac{\left<W\right>^{2\beta}}
{\left<s^\Vert\right>\left<s^{\Vert}+W\right>^{2\beta}
\left<s^{\bot}+W\right>^{2\beta-1+\delta}}
\end{equation}
Therefore, for all large enough $|W|$,
\begin{equation}
\begin{aligned}
I_3^\prime & \lesssim \int_{\mathbb{S}^{d-1}} d\omega \int_{\mathbb{R}^d} ds
\frac{\left<W\right>^{2\beta}}
{\left< s^\Vert \right>
\left<s^{\Vert}+W\right>^{2\beta}
\left<s^{\bot}+W\right>^{2\beta-1+\delta}} \\
& = \int_{\mathbb{S}^{d-1}} d\omega \left< W \right>^{2\beta}
\left( \int \frac{ds^\Vert}
{\left< s^\Vert \right>
\left< s^\Vert + W \right>^{2\beta}}\right)
\left( \int \frac{ds^\bot}
{\left< s^\bot + W \right>^{2\beta-1+\delta}}\right) \\
& \lesssim \int_{\mathbb{S}^{d-1}} d\omega
\left< W \right>^{2\beta} \left(
\left< W\right>^{-1} \left<W^\bot\right>^{1-2\beta}
\log \left<W\right> \right)
\left( \left<W^\Vert\right>^{d-2\beta-\delta}\right)\\
\end{aligned}
\end{equation}
As before, the integral over $s^\bot$ is estimated by a trivial
computation, whereas the integral over $s^\Vert$ may be estimated
by considering separately the regions $|s^\Vert| < \frac{1}{2} |W|$,
$|s^\Vert| > 2|W|$, and $\frac{1}{2} |W| \leq |s^\Vert| \leq 2|W|$.

We find that $I_3^\prime$ obeys the following
estimate:
\begin{equation}
I_3^\prime \lesssim
 \int_{\mathbb{S}^{d-1}} d\omega
\left< W \right>^{2\beta-1+\frac{1}{2}\delta}
 \left<W^\bot\right>^{1-2\beta}
\left<W^\Vert\right>^{d-2\beta-\delta}
\end{equation}
Then we have
\begin{equation}
\left< W \right>^{2\beta-1+\frac{1}{2}\delta} \lesssim
\left< W^\Vert \right>^{2\beta-1+\frac{1}{2}\delta} +
\left< W^\bot \right>^{2\beta-1+\frac{1}{2}\delta}
\end{equation}
Hence $I_3^\prime \lesssim I_3^{\prime \prime} + 
I_3^{\prime \prime \prime}$ where
\begin{equation}
I_3^{\prime \prime} = \int_{\mathbb{S}^{d-1}} d\omega
\left< W^\bot\right>^{1-2\beta}
\left< W^\Vert\right>^{d-1-\frac{1}{2}\delta}
\end{equation}
\begin{equation}
I_3^{\prime \prime \prime} = \int_{\mathbb{S}^{d-1}}
d\omega \left< W^\bot\right>^{\frac{1}{2} \delta}
\left< W^\Vert \right>^{d-2\beta-\delta}
\end{equation}
Then for any sufficiently small $\delta$ and
$\beta > d$, both $I_3^{\prime \prime}$
and $I_3^{\prime \prime \prime}$ may be bounded using dyadic
decompositions in the angular parameter $\omega$, as follows:
neglecting additive constants,
\begin{equation}
\begin{aligned}
I_3^{\prime \prime} & \lesssim \sum_{k=1}^\infty
\int_{\omega : 2^{-k-1} |W^\Vert| \leq |W^\bot| < 2^{-k} |W^\Vert|}
d\omega \left< W^\bot\right>^{1-2\beta}
\left< W^\Vert\right>^{d-1-\frac{1}{2}\delta } \\
& \lesssim \sum_{k=1}^\infty 2^{-k-1} \times
(2^{-k})^{d-2} \times (2^{k+1})^{d-1-\frac{1}{2}\delta} <\infty
\end{aligned}
\end{equation}
\begin{equation}
\begin{aligned}
I_3^{\prime \prime \prime} & \lesssim \sum_{k=1}^\infty
\int_{\omega : 2^{-k-1} |W^\bot| \leq |W^\Vert| < 2^{-k} |W^\bot|}
d\omega \left< W^\bot\right>^{\frac{1}{2}\delta}
\left< W^\Vert\right>^{d-2\beta-\delta} \\
& \lesssim \sum_{k=1}^\infty 2^{-k-1} \times (2^{k+1})^{\frac{1}{2}\delta}
 <\infty
\end{aligned}
\end{equation}
The factor of $(2^{-k})^{d-2}$ in $I_3^{\prime \prime}$ comes from
the Jacobian for spherical coordinates in $\mathbb{R}^d$.

We finally turn to $I_1$, which is clearly bounded by the following
quantity:
\begin{equation}
I_1 \lesssim \sup_{W\in\mathbb{R}^d} 
\sup_{P\subset \mathbb{R}^d : \dim P = d-1}
\int_{P} dS(s) \frac{\left<W\right>^{2\alpha}}
{\left<s\right>^{2\alpha} \left<s+W\right>^{2\alpha}}
\end{equation}
The integrals over $P\cap \left\{ |s|<\frac{1}{2} |W|\right\}$,
 $P\cap \left\{ |s| > 2 |W|\right\}$, and
$P \cap \left\{ \frac{1}{2} |W| \leq |s| \leq 2|W|\right\}$
 are each easily bounded uniformly
in $W$ as long as $\alpha > \frac{d-1}{2}$.

To summarize, as long as $\alpha > \frac{d-1}{2}$, $\beta > d$,
and $\max (0,2A-1) < \frac{1}{\sigma} \leq 2$, then for 
$r\in [0,1)$ such that
$\frac{r}{\sigma} \geq \max(0,2A-1+\delta)$ for a small $\delta>0$
 we have for any $\kappa_0 > \kappa > 0$ the following estimate:
\begin{equation}
\begin{aligned}
& \left\Vert B^+_{i,k+1} \left[ 
e^{\frac{1}{2} i t \left( \Delta_{X_{k+1}} - \Delta_{X_{k+1}^\prime}\right)}
\gamma_0^{(k+1)} \right] \right\Vert_{L^2_t 
H^{\alpha,\beta,\sigma,\kappa}_k } \leq \\
& \qquad \qquad \qquad \qquad  \leq
C(\alpha,\beta,\sigma,r) \left\Vert \mathbf{b} \right\Vert_{L^\infty_A}
\left(\kappa_0-\kappa\right)^{-\frac{1}{2}r}
 \left\Vert \gamma_0^{(k+1)} 
\right\Vert_{H^{\alpha,\beta,\sigma,\kappa_0}_{k+1}}
\end{aligned}
\end{equation}

\section{Proof of Theorem \ref{thm:BE-LWP}}
\label{sec:3}

Formally speaking, solutions of Boltzmann's equation are factorized
solutions of the Boltzmann hierarchy, i.e. 
$\gamma^{(k)} = \gamma^{\otimes k}$. We use the notation
\begin{equation}
\Delta_{\pm}^{(k)} = \Delta_{X_{k}}-\Delta_{X_k^\prime}
\end{equation}
\begin{equation}
\Delta_{\pm} = \Delta_{\pm}^{(1)}
\end{equation}
Then if $B_{k+1} = \sum_{i=1}^k \left( B_{i,k+1}^+ - B_{i,k+1}^-\right)$,
 the Boltzmann hierarchy in integral form reads as follows:
\begin{equation}
\gamma^{(k)} (t) = e^{\frac{1}{2} i t \Delta_{\pm}^{(k)}} 
\gamma^{(k)} (0) - i \int_0^t 
e^{\frac{1}{2} i (t-t_1) \Delta_{\pm}^{(k)}}
B_{k+1} \gamma^{(k+1)} (t_1) dt_1
\end{equation}
Let us assume $\gamma^{(k)} = \gamma^{\otimes k}$ for all $k\in\mathbb{N}$
and consider the Boltzmann hierarchy for $k=1,2$:
\begin{equation}
\gamma (t) = e^{\frac{1}{2} i t \Delta_{\pm}} 
\gamma (0)  - i \int_0^t 
e^{\frac{1}{2} i (t-t_1) \Delta_{\pm}}
B_{2} \left( \gamma^{\otimes 2}\right) (t_1) dt_1
\end{equation}
\begin{equation}
\left(\gamma\otimes \gamma\right) (t) =
 e^{\frac{1}{2} i t \Delta_{\pm}^{(2)}} 
\left( \gamma\otimes\gamma\right) (0) - i\int_0^t 
e^{\frac{1}{2} i (t-t_1) \Delta_{\pm}^{(2)}}
B_{3} \left( \gamma^{\otimes 3}\right) (t_1) dt_1
\end{equation}
Now we apply the operator $B_2$ to the second equation, thereby
obtaining the following system:
\begin{equation}
\gamma (t) = e^{\frac{1}{2} i t \Delta_{\pm}} 
\gamma (0) - i \int_0^t 
e^{\frac{1}{2} i (t-t_1) \Delta_{\pm}}
B_{2} \left( \gamma^{\otimes 2}\right) (t_1) dt_1
\end{equation}
\begin{equation}
\begin{aligned}
& B_2 \left(\gamma^{\otimes 2}\right) (t) = \\
& \qquad  = B_2\left( e^{\frac{1}{2} i t \Delta_{\pm}^{(2)}} 
\left( \gamma^{\otimes 2} \right) (0)\right) - i \int_0^t 
B_2 \left[ e^{\frac{1}{2} i (t-t_1) \Delta_{\pm}^{(2)}}
B_{3} \left( \gamma^{\otimes 3}\right) (t_1) \right] dt_1
\end{aligned}
\end{equation}
Let us observe that $B(\gamma_1,\gamma_2) = 
B_2 (\gamma_1 \otimes \gamma_2)$. Therefore if we define
$\zeta (t) = B \left( \gamma(t),\gamma(t)\right)$ then we obtain
the following system of equations for the pair 
$\left( \gamma,\zeta\right)$:
\begin{equation}
\gamma (t) = e^{\frac{1}{2} i t \Delta_{\pm}} 
\gamma (0) - i \int_0^t 
e^{\frac{1}{2} i (t-t_1) \Delta_{\pm}}
\zeta (t_1) dt_1
\end{equation}
\begin{equation}
\begin{aligned}
\zeta  (t) =
& B \left( e^{\frac{1}{2} i t \Delta_{\pm}}\gamma (0),
e^{\frac{1}{2} i t \Delta_{\pm} }\gamma(0)\right) +\\
& \qquad \quad + (-i) \int_0^t 
B \left( e^{\frac{1}{2} i (t-t_1) \Delta_{\pm}}\gamma(t_1),
e^{\frac{1}{2} i (t-t_1) \Delta_{\pm}} \zeta(t_1)\right) dt_1 +\\
& \qquad \quad +  (-i) \int_0^t 
B \left( e^{\frac{1}{2} i (t-t_1) \Delta_{\pm}}\zeta(t_1),
e^{\frac{1}{2} i (t-t_1) \Delta_{\pm}} \gamma(t_1)\right) dt_1
\end{aligned}
\end{equation}
We will solve this simultaneous system of equations for
$(\gamma(t),\zeta(t))$ on a small time interval $[0,T]$ by Picard
iteration, using the following norm:
\begin{equation}
\left\Vert (\gamma,\zeta)\right\Vert =
T^{\frac{1}{2} (1-r)}  \left\Vert \left\Vert \gamma (t) 
\right\Vert_{H^{\alpha,\beta,\sigma,\kappa-\lambda t}}
\right\Vert_{L^\infty_T}+
\left\Vert \left\Vert
\zeta(t)\right\Vert_{H^{\alpha,\beta,\sigma,\kappa-\lambda t}}
\right\Vert_{L^1_T}
\end{equation}
Here we have fixed some $r\in [0,1)$ as in the statement of
Proposition \ref{prop:spacetime-est}.
The key result we will use is that Proposition
\ref{prop:spacetime-est} implies the following bilinear estimates:
\begin{equation}
\label{eq:bilinear}
\begin{aligned}
& \left\Vert B \left( e^{\frac{1}{2} i t \Delta_{\pm}} \gamma_{0,1},
e^{\frac{1}{2} i t \Delta_{\pm}} \gamma_{0,2} \right) 
\right\Vert_{L^2_t H^{\alpha,\beta,\sigma,\kappa_1}} \leq \\
& \qquad \qquad \qquad \qquad 
\leq C \left( 1 + (\kappa_0-\kappa_1)^{-\frac{1}{2} r}\right)
\left\Vert \gamma_{0,1} \right\Vert_{H^{\alpha,\beta,\sigma,\kappa_0}}
\left\Vert \gamma_{0,2} \right\Vert_{H^{\alpha,\beta,\sigma,\kappa_0}}
\end{aligned}
\end{equation}

To set up the fixed point iteration, we fix the initial data
$\gamma_0 \in H^{\alpha,\beta,\sigma,\kappa}$ and define the map
$\Phi = \left(\Phi_1,\Phi_2\right) (\gamma,\zeta)$ as follows:
\begin{equation}
\label{eq:Phi1}
\left[ \Phi_1 (\gamma,\zeta)\right]
 (t) = e^{\frac{1}{2} i t \Delta_{\pm}} 
\gamma_0 - i \int_0^t 
e^{\frac{1}{2} i (t-t_1) \Delta_{\pm}}
\zeta (t_1) dt_1
\end{equation}
\begin{equation}
\label{eq:Phi2}
\begin{aligned}
\left[ \Phi_2 (\gamma,\zeta)\right]  (t) =
& B \left( e^{\frac{1}{2} i t \Delta_{\pm}}\gamma_0,
e^{\frac{1}{2} i t \Delta_{\pm} }\gamma_0\right) +\\
& \qquad + (-i) \int_0^t 
B \left( e^{\frac{1}{2} i (t-t_1) \Delta_{\pm}}\gamma(t_1),
e^{\frac{1}{2} i (t-t_1) \Delta_{\pm}} \zeta(t_1)\right) dt_1 +\\
& \qquad  +  (-i) \int_0^t 
B \left( e^{\frac{1}{2} i (t-t_1) \Delta_{\pm}}\zeta(t_1),
e^{\frac{1}{2} i (t-t_1) \Delta_{\pm}} \gamma(t_1)\right) dt_1
\end{aligned}
\end{equation}
We wish to solve the equation $(\gamma,\zeta) = \Phi (\gamma,\zeta)$.

First, using (\ref{eq:Phi1}) and the fact that the propagator
$e^{\frac{1}{2} i t \Delta_{\pm}}$ preserves the space
$H^{\alpha,\beta,\sigma,\kappa}$, along with the embedding
$H^{\alpha,\beta,\sigma,\kappa_0} \subset
H^{\alpha,\beta,\sigma,\kappa_1}$ for $\kappa_0 > \kappa_1 > 0$, we easily
obtain:
\begin{equation}
\label{eq:Phi1-bound}
\left\Vert\left\Vert \left[ \Phi_1 (\gamma,\zeta)\right](t)
\right\Vert_{H^{\alpha,\beta,\sigma,\kappa-\lambda t}}
\right\Vert_{L^\infty_T} \leq
\left\Vert \gamma_0 \right\Vert_{H^{\alpha,\beta,\sigma,\kappa}}
+ \left\Vert \left\Vert \zeta(t) 
\right\Vert_{H^{\alpha,\beta,\sigma,\kappa-\lambda t}}
\right\Vert_{L^1_T}
\end{equation}

We now turn to $\Phi_2$. We begin by estimating the first term on
the right hand side of (\ref{eq:Phi2}).  We will use a dyadic 
decomposition in time:
\begin{equation*}
\begin{aligned}
& \left\Vert \left\Vert
B\left( e^{\frac{1}{2} i t \Delta_{\pm}} \gamma_0,
e^{\frac{1}{2} i t \Delta_{\pm}} \gamma_0\right)
\right\Vert_{H^{\alpha,\beta,\sigma,\kappa-\lambda t}}
\right\Vert_{L^1_T} = \\
& =  \sum_{m=0}^\infty \int_{2^{-m-1} T < t \leq 2^{-m} T}
\left\Vert B \left( e^{\frac{1}{2} i t \Delta_{\pm}} \gamma_0,
e^{\frac{1}{2} i t \Delta_{\pm}} \gamma_0\right)
\right\Vert_{H^{\alpha,\beta,\sigma,\kappa-\lambda t}} dt\\
& \leq  \sum_{m=0}^\infty \int_{2^{-m-1} T < t \leq 2^{-m} T}
\left\Vert B \left( e^{\frac{1}{2} i t \Delta_{\pm}} \gamma_0,
e^{\frac{1}{2} i t \Delta_{\pm}} \gamma_0\right)
\right\Vert_{H^{\alpha,\beta,\sigma,\kappa-\lambda 2^{-m-1}T}} dt
\end{aligned}
\end{equation*}
Now apply the Cauchy-Schwarz inequality, followed by
(\ref{eq:bilinear}). We implicitly assume $\lambda T < 1$, which is
acceptable because we only want to address small times $T$ in any case.
\begin{equation*}
\begin{aligned}
& \left\Vert \left\Vert
B\left( e^{\frac{1}{2} i t \Delta_{\pm}} \gamma_0,
e^{\frac{1}{2} i t \Delta_{\pm}} \gamma_0\right)
\right\Vert_{H^{\alpha,\beta,\sigma,\kappa-\lambda t}}
\right\Vert_{L^1_T} \leq \\
& \leq  \sum_{m=0}^\infty 
\left( 2^{-m-1} T\right)^{\frac{1}{2}}
\left\Vert B \left( e^{\frac{1}{2} i t \Delta_{\pm}} \gamma_0,
e^{\frac{1}{2} i t \Delta_{\pm}} \gamma_0\right)
\right\Vert_{L^2_t H^{\alpha,\beta,\sigma,\kappa-\lambda 2^{-m-1}T}}\\
& \leq \sum_{m=0}^\infty
\left( 2^{-m-1}T\right)^{\frac{1}{2}}
\frac{C}{\left(\lambda 2^{-m-1} T\right)^{\frac{1}{2}r}}
\left\Vert \gamma_0 \right\Vert_{H^{\alpha,\beta,\sigma,\kappa}}^2\\
& \leq C\lambda^{-\frac{1}{2} r} T^{\frac{1}{2} (1-r)}
\left( \sum_{m=0}^\infty 2^{-\frac{1}{2} m (1-r)}\right)
\left\Vert \gamma_0 \right\Vert_{H^{\alpha,\beta,\sigma,\kappa}}^2
\end{aligned}
\end{equation*}
We now estimate the second term on the right hand side of
(\ref{eq:Phi2}); the third term is handled similarly. We will employ
a dyadic decomposition in $t-t_1$ and apply Cauchy-Schwarz and
(\ref{eq:bilinear}) as before.
\begin{equation*}
\begin{aligned}
& \left\Vert \left\Vert \int_0^t
B\left( e^{\frac{1}{2} i (t-t_1)\Delta_{\pm}} \gamma(t_1),
e^{\frac{1}{2} i (t-t_1) \Delta_{\pm}} \zeta(t_1)\right) dt_1
\right\Vert_{H^{\alpha,\beta,\sigma,\kappa-\lambda t}} \right\Vert_{L^1_T}
\leq \\
& \leq \sum_{m=0}^\infty \int_0^T  dt_1
\int_{t_1+2^{-m-1}T < t \leq t_1+2^{-m} T}dt \times \\
& \qquad  \qquad \times
 \left\Vert B \left( e^{\frac{1}{2} i (t-t_1) \Delta_{\pm}}\gamma(t_1),
e^{\frac{1}{2} i (t-t_1) \Delta_{\pm}} \zeta(t_1)\right)
\right\Vert_{H^{\alpha,\beta,\sigma,\kappa-\lambda t}} \\
& \leq  \sum_{m=0}^\infty \int_0^T  dt_1
\int_{t_1+2^{-m-1}T < t \leq t_1+2^{-m} T}dt \times \\
& \qquad  \qquad  \times
 \left\Vert B \left( e^{\frac{1}{2} i (t-t_1) \Delta_{\pm}}\gamma(t_1),
e^{\frac{1}{2} i (t-t_1) \Delta_{\pm}} \zeta(t_1)\right)
\right\Vert_{H^{\alpha,\beta,\sigma,\kappa-\lambda (t_1+2^{-m-1}T)}} \\
& \leq  \sum_{m=0}^\infty \left( 2^{-m-1} T\right)^{\frac{1}{2}}
\int_0^T  dt_1
 \times \\
& \qquad  \qquad  \times
 \left\Vert B \left( e^{\frac{1}{2} i (t-t_1) \Delta_{\pm}}\gamma(t_1),
e^{\frac{1}{2} i (t-t_1) \Delta_{\pm}} \zeta(t_1)\right)
\right\Vert_{L^2_t
 H^{\alpha,\beta,\sigma,\kappa-\lambda (t_1+2^{-m-1}T)}} \\
& \leq  \sum_{m=0}^\infty \left( 2^{-m-1} T\right)^{\frac{1}{2}}
\times C \left( \lambda 2^{-m-1} T \right)^{-\frac{1}{2}r}\times\\
& \qquad \qquad \times \int_0^T  dt_1
 \left\Vert \gamma(t_1) \right\Vert_{H^{\alpha,\beta,\sigma,
\kappa-\lambda t_1}}
\left\Vert \zeta(t_1) \right\Vert_{H^{\alpha,\beta,\sigma,
\kappa-\lambda t_1}} \\
& \leq C \lambda^{-\frac{1}{2} r}
T^{\frac{1}{2} (1-r)} \left( \sum_{m=0}^\infty
2^{-\frac{1}{2} m (1-r)} \right) \times \\
& \qquad \qquad \times 
\left\Vert \left\Vert \gamma(t) \right\Vert_{H^{\alpha,\beta,\sigma,
\kappa-\lambda t}} \right\Vert_{L^\infty_T}
\left\Vert \left\Vert \zeta(t) \right\Vert_{H^{\alpha,\beta,\sigma,
\kappa-\lambda t}} \right\Vert_{L^1_T}
\end{aligned}
\end{equation*}
We can finally conclude the following estimate for $\Phi_2$:
\begin{equation}
\label{eq:Phi2-bound}
\begin{aligned}
& \left\Vert \left\Vert \left[ \Phi_2 (\gamma,\zeta)\right] (t)
\right\Vert_{H^{\alpha,\beta,\sigma,\kappa-\lambda t}}
\right\Vert_{L^1_T} \leq
C \lambda^{-\frac{1}{2} r} T^{\frac{1}{2} (1-r)}
\left( \sum_{m=0}^\infty 2^{-\frac{1}{2} m (1-r)}\right) \times \\
& \times \left( \left\Vert \gamma_0 \right\Vert_{H^{\alpha,\beta,\sigma,
\kappa}}^2 +
\left\Vert \left\Vert \gamma(t) \right\Vert_{H^{\alpha,\beta,\sigma,
\kappa-\lambda t}} \right\Vert_{L^\infty_T}
\left\Vert \left\Vert \zeta(t) \right\Vert_{H^{\alpha,\beta,\sigma,
\kappa-\lambda t}} \right\Vert_{L^1_T}\right)
\end{aligned}
\end{equation}

Combining (\ref{eq:Phi1-bound}) and (\ref{eq:Phi2-bound}), and defining
$C_{\lambda,r} = C \lambda^{-\frac{1}{2} r} \sum_{m=0}^\infty
2^{-\frac{1}{2} m (1-r)}$, we obtain:
\begin{equation}
\label{eq:Phi-bound}
\begin{aligned}
& \left\Vert \Phi (\gamma,\zeta)\right\Vert \leq
T^{\frac{1}{2} (1-r)} \left\Vert \gamma_0 \right\Vert_{H^{\alpha,\beta,
\sigma,\kappa}} + C_{\lambda,r}
T^{\frac{1}{2} (1-r)} \left\Vert \gamma_0
\right\Vert_{H^{\alpha,\beta,\sigma,\kappa}}^2 + \\
& \qquad \qquad \qquad \qquad \qquad \qquad 
+ T^{\frac{1}{2} (1-r)} \left\Vert (\gamma,\zeta)\right\Vert +
C_{\lambda,r} \left\Vert (\gamma,\zeta)\right\Vert^2
\end{aligned}
\end{equation}
By a completely analogous argument, we obtain the following
continuity bound:
\begin{equation}
\label{eq:Phi-continuity}
\begin{aligned}
& \left\Vert \Phi (\gamma_1,\zeta_1) - \Phi (\gamma_2,
\zeta_2)\right\Vert \leq \\
& \qquad 
\leq \left( T^{\frac{1}{2} (1-r)} + 4 
C_{\lambda,r} \left( \left\Vert (\gamma_1,\zeta_1)\right\Vert +
\left\Vert (\gamma_2,\zeta_2) \right\Vert \right)
\right) \left\Vert (\gamma_1,\zeta_1) - 
(\gamma_2,\zeta_2) \right\Vert
\end{aligned}
\end{equation}
Combining (\ref{eq:Phi-bound}) and (\ref{eq:Phi-continuity}),
and applying the Banach fixed point theorem, we conclude the existence
and uniqueness of a solution to the fixed point equation
$(\gamma,\zeta)=\Phi(\gamma,\zeta)$ once $T$ is chosen sufficiently
small depending only on $\left\Vert \gamma_0 
\right\Vert_{H^{\alpha,\beta,\sigma,\kappa}}$. This gives us 
uniqueness under the assumption that $\left\Vert (\gamma,\zeta)\right\Vert$
is small, but in fact for any solution we can apply
(\ref{eq:Phi-bound}) and a standard continuity argument to conclude
that $\left\Vert (\gamma,\zeta)\right\Vert$ is necessarily small if
$T$ is small, as long as it is finite for some positive $T$.
The estimate (\ref{eq:BE-est}) follows directly from
(\ref{eq:Phi-bound}).

Finally we remark that if $A \in \left[ 0,\frac{1}{2}\right)$ then
we may take $r=0$, so that $C_{\lambda,r}$ loses its dependence on
$\lambda$; hence, we are allowed to take $\lambda = 0$ and we can permit
any $\kappa \in [0,\infty)$. The rest of
the proof proceeds in exactly the same manner.

\section{Proof of Theorem \ref{thm:BH-LWP}}
\label{sec:4}

Theorem \ref{thm:BH-LWP} follows from Proposition \ref{prop:spacetime-est}
combined with the arguments of Chen and Pavlovi{\' c}, \cite{PC2010},
which in turn rely on
the combinatorial arguments of
Erd{\" o}s-Schlein-Yau, \cite{ES2001,ESY2006,ESY2007}, in the boardgame
representation given by Klainerman and Machedon in \cite{KM2008}. We
outline the proof here for the reader's convenience.

To begin, we point out that the Boltzmann hierarchy may be written
in vector integral form as follows:
\begin{equation}
\Gamma (t) = e^{\frac{1}{2} i t \hat{\Delta}_{\pm}} \Gamma(0)
- i \int_0^t e^{\frac{1}{2}i (t-t_1) \hat{\Delta}_{\pm}}
B \Gamma (t_1) dt_1
\end{equation}
where $\hat{\Delta}_\pm \Gamma = \left\{
\left( \Delta_{X_k}-\Delta_{X_k^\prime}\right) \gamma^{(k)}
\right\}_{k\in\mathbb{N}}$ and
$B\Gamma = \left\{ B_{k+1} \gamma^{(k+1)} \right\}_{k\in\mathbb{N}}$.
We can apply $B$ to both sides to yield a closed equation for
$B\Gamma$:
\begin{equation}
B\Gamma (t) = B\left[e^{\frac{1}{2} i t \hat{\Delta}_{\pm}} \Gamma(0)\right]
- i\int_0^t B\left[ e^{\frac{1}{2} i(t-t_1) \hat{\Delta}_{\pm}}
B \Gamma (t_1)\right] dt_1
\end{equation}
Letting $\Xi = B\Gamma$, we conclude that the pair
$\left( \Gamma, \Xi\right)$ satisfies the following system of
equations:
\begin{equation}
\label{eq:Gamma-1}
\Gamma (t) = e^{\frac{1}{2} i t \hat{\Delta}_{\pm}} \Gamma(0)
- i \int_0^t e^{\frac{1}{2}i (t-t_1) \hat{\Delta}_{\pm}}
\Xi (t_1) dt_1
\end{equation}
\begin{equation}
\label{eq:Xi-1}
\Xi (t) = B\left[e^{\frac{1}{2} i t \hat{\Delta}_{\pm}} \Gamma(0)\right]
- i \int_0^t B\left[ e^{\frac{1}{2} i(t-t_1) \hat{\Delta}_{\pm}}
\Xi (t_1)\right] dt_1
\end{equation}
and this system is equivalent to the original Boltzmann hierarchy.

Since (\ref{eq:Xi-1}) is a closed equation for $\Xi$, we proceed in
two steps. First we solve (\ref{eq:Xi-1}) on a small time interval
$[0,T]$ by Picard iteration; then, we establish that the right-hand
side of (\ref{eq:Gamma-1}) is indeed well-defined in the correct
functional space. 
The proof proceeds by iterating the Duhamel formula
(\ref{eq:Xi-1}), $k$ times for the $k$th component,
and applying the combinatorial methods of
Erd{\" o}s, Schlein and Yau,
\cite{ES2001,ESY2006,ESY2007}, expressed in boardgame form
by Klainerman and Machedon \cite{KM2008}.  Then we conclude by
applying Proposition \ref{prop:spacetime-est} inductively to
bound all the terms (which are now $\mathcal{O}(C^k)$ in number
instead of (even more than)
$ \mathcal{O}(k!)$ due to the combinatorial methods of
\cite{ES2001,ESY2006,ESY2007}). The precise details are written
out in \cite{PC2010} for the interested reader.

\begin{remark}
Note that if $A\geq \frac{1}{2}$, then the combinatorial methods
of \cite{ES2001,ESY2006,ESY2007}, and the reformulation in terms
of a boardgame argument \cite{KM2008}, all still apply at the formal
level. However, the termwise estimates of \cite{PC2010} no longer apply
uniformly across general
 re-ordering of collision times. This is simply not an
issue when $A<\frac{1}{2}$ since no time-dependent
 loss of weight is required in that
case. Most likely, if LWP holds at all for the hard sphere Boltzmann
hierarchy for the spaces we consider, then completely new estimates
(different from Proposition \ref{prop:spacetime-est}) will be
required.
\end{remark}

\begin{appendix}
\section{Inverse Wigner Transform of the Boltzmann Equation and 
Boltzmann Hierarchy}
\label{sec:appendix-A}

We begin with the Boltzmann hierarchy.

\begin{proposition}
\label{prop:App-BH-TR}
Let $f^{(k)} \in L^1 \left( [0,T],
L^2 (\mathbb{R}^{dk} \times \mathbb{R}^{dk})\right)$ and let
$\gamma^{(k)}$ denote the inverse Wigner transform of $f^{(k)}$.
Then if
\begin{equation}
\left( \partial_t + V_k \cdot \nabla_{X_k}\right) f^{(k)} =
g^{(k)}
\end{equation}
holds in the sense of distributions, then we have
\begin{equation}
\begin{aligned}
& \left( i \partial_t + \frac{1}{2} \left(\Delta_{X_k} -
\Delta_{X_k^\prime}\right)\right) \gamma^{(k)} (t,X_k,X_k^\prime) = \\
& \qquad \qquad \qquad \qquad \qquad = 
i \int_{\mathbb{R}^{dk}} g^{(k)} \left( t,
\frac{X_k + X_k^\prime}{2},V_k\right) 
e^{i V_k \cdot (X_k - X_k^\prime)} dV_k
\end{aligned}
\end{equation}
in the sense of distributions.
\end{proposition}
\begin{proof}
We have
\begin{equation*}
\begin{aligned}
& i \partial_t \gamma^{(k)} (t,X_k,X_k^\prime) = 
i \int_{\mathbb{R}^{dk}} \left( \partial_t f^{(k)}\right)\left( 
t,\frac{X_k+X_k^\prime}{2},V_k\right) e^{i V_k \cdot (X_k-X_k^\prime)} dV_k \\
& = i \int_{\mathbb{R}^{dk}} 
\left( -V_k\cdot \nabla_{X_k} f^{(k)} + g^{(k)} \right)
\left( t,\frac{X_k+X_k^\prime}{2},V_k\right) e^{iV_k\cdot(X_k-X_k^\prime)} dV_k
\end{aligned}
\end{equation*}
Consider the transport term alone.
\begin{equation*}
\begin{aligned}
& i \int 
\left( -V_k\cdot \nabla_{X_k} f^{(k)} \right)
\left( t,\frac{X_k+X_k^\prime}{2},V_k\right)e^{iV_k\cdot(X_k-X_k^\prime)}dV_k\\
& = - i \int
V_k \cdot ( \nabla_{X_k} + \nabla_{X_k^\prime} ) \left[ f^{(k)}
\left( t,\frac{X_k+X_k^\prime}{2},V_k\right)\right]
 e^{iV_k\cdot (X_k-X_k^\prime)} dV_k \\
& = - \int
( \nabla_{X_k} + \nabla_{X_k^\prime} ) \left[ f^{(k)}
\left( t,\frac{X_k+X_k^\prime}{2},v\right)\right] \cdot
i V_k e^{iV_k\cdot (X_k-X_k^\prime)} dV_k \\
& = - ( \nabla_{X_k} + \nabla_{X_k^\prime} ) \cdot \int
f^{(k)} \left( t,\frac{X_k+X_k^\prime}{2},V_k\right) 
i V_k e^{iV_k\cdot (X_k-X_k^\prime)} dV_k \\
& = - ( \nabla_{X_k} + \nabla_{X_k^\prime} ) \cdot \int
f^{(k)} \left( t,\frac{X_k+X_k^\prime}{2},V_k\right) 
\frac{\nabla_{X_k} - \nabla_{X_k^\prime}}{2}e^{iV_k\cdot(X_k-X_k^\prime)}dV_k\\
& = - ( \nabla_{X_k} + \nabla_{X_k^\prime} ) \cdot
\frac{\nabla_{X_k} - \nabla_{X_k^\prime}}{2}  \int 
f^{(k)} \left( t,\frac{X_k+X_k^\prime}{2},V_k\right) 
 e^{iV_k\cdot (X_k-X_k^\prime)} dV_k \\
& = -\frac{1}{2} \left( \Delta_{X_k} - \Delta_{X_k^\prime} \right)
\gamma^{(k)} (t,X_k,X_k^\prime)
\end{aligned}
\end{equation*}
Therefore,
\begin{equation}
\begin{aligned}
&\left( i \partial_t + \frac{1}{2}\left(\Delta_{X_k}-\Delta_{X_k^\prime}\right)
\right) \gamma^{(k)} (t,X_k,X_k^\prime) =\\
&\qquad \qquad \qquad =i \int_{\mathbb{R}^{dk}} g^{(k)}
 \left( t,\frac{X_k+X_k^\prime}{2},V_k\right)
e^{iV_k\cdot (X_k-X_k^\prime)} dV_k
\end{aligned}
\end{equation}
\end{proof}

\begin{proposition}
\label{prop:App-BH-Coll}
Let $f^{(k+1)} (X_{k+1},V_{k+1})$ be a Schwartz function, and let
$\gamma^{(k+1)}$ denote its inverse Wigner transform.
Then
\begin{equation}
\begin{aligned}
& i\int C^-_{1,k+1} f^{(k+1)} \left( \frac{X_k+X_k^\prime}{2},V_k\right)
e^{i V_k\cdot (X_k-X_k^\prime)} dV_k = \\
& = \frac{i}{ 2^{2d} \pi^{d}}
\int_{\mathbb{S}^{d-1}} d\omega \int_{\mathbb{R}^d} dz 
\hat{\mathbf{b}}^\omega
\left(\frac{z}{2}\right)
 \times \\
& \qquad \times \gamma^{(k+1)} \left( x_1-\frac{z}{4},X_{2:k},
\frac{x_1+x_1^\prime}{2}+\frac{z}{4}, x_1^\prime + \frac{z}{4},X_{2:k}^\prime,
\frac{x_1+x_1^\prime}{2}-\frac{z}{4} \right) \\
\end{aligned}
\end{equation}
and
\begin{equation}
\begin{aligned}
& i \int C^+_{1,k+1} f^{(k+1)} \left( \frac{X_k+X_k^\prime}{2},V_k\right)
e^{iV_k \cdot (X_k-X_k^\prime)} dV_k = \\ 
& =  \frac{i}{2^{2d} \pi^{d}}  \int_{\mathbb{S}^{d-1}}
 d\omega \int_{\mathbb{R}^d} dz
 \hat{\mathbf{b}}^\omega \left( \frac{z}{2} \right)  \times \\
& \qquad \times \gamma^{(k+1)} \left( x_1-\frac{1}{2} P_\omega
(x_1-x_1^\prime)-\frac{R_\omega (z)}{4}, X_{2:k}, \right. \\
& \qquad \qquad \qquad \qquad \qquad \qquad \left.
\frac{x_1+x_1^\prime}{2}+\frac{1}{2} P_\omega
(x_1-x_1^\prime)+\frac{R_\omega (z)}{4},\right.\\
& \qquad \qquad  \left. 
x_1^\prime + \frac{1}{2} P_\omega (x_1-x_1^\prime)+
\frac{R_\omega (z)}{4}, X_{2:k}^\prime, \right. \\
& \qquad \qquad \qquad \qquad \qquad \qquad \left.
\frac{x_1+x_1^\prime}{2}-\frac{1}{2} P_\omega
(x_1-x_1^\prime)-\frac{R_\omega (z)}{4}\right)  \\
\end{aligned}
\end{equation}
\end{proposition}
\begin{proof}
Consider the loss term.
\begin{equation*}
\begin{aligned}
& i\int C^-_{1,k+1} f^{(k+1)} \left( \frac{X_k+X_k^\prime}{2},V_k\right)
e^{iV_k\cdot (X_k-X_k^\prime)} dV_k = \\
& = i \int
dV_{k+1} d\omega \mathbf{b} \left( |v_{k+1}-v_1|, \omega\cdot\frac{v_{k+1}-v_1}
{|v_{k+1}-v_1|} \right) \times \\
& \qquad \qquad  \times  e^{i V_k\cdot (X_k-X_k^\prime)}
f^{(k+1)} \left( \frac{X_k+X_k^\prime}{2},
\frac{x_1+x_1^\prime}{2},V_k,v_{k+1} \right) \\
& = i \frac{1}{(2\pi)^{d(k+1)}}
\int dV_{k+1} dY_{k+1} d\omega \mathbf{b} \left( |v_{k+1}-v_1|,
\omega\cdot\frac{v_{k+1}-v_1} {|v_{k+1}-v_1|} \right)\times \\
& \times e^{i V_k\cdot (X_k-X_k^\prime)}
e^{-i V_{k+1}\cdot Y_{k+1}} \times \\
& \times \gamma^{(k+1)} \left( \frac{X_k+X_k^\prime+Y_k}{2},
\frac{x_1+x_1^\prime+y_{k+1}}{2},
\frac{X_k+X_k^\prime-Y_k}{2},
\frac{x_1+x_1^\prime-y_{k+1}}{2} \right) \\
& = i \frac{1}{(2\pi)^{2d}}
\int dv_1 dv_{k+1} dy_1 dy_{k+1} d\omega 
\mathbf{b} \left( |v_{k+1}-v_1|, \omega\cdot\frac{v_{k+1}-v_1}
{|v_{k+1}-v_1|} \right) \times \\
& \times e^{i v_1 \cdot (x_1-x_1^\prime-y_1)}
e^{-i v_{k+1}\cdot y_{k+1}} \times \\
& \times \gamma^{(k+1)} \left( \frac{x_1+x_1^\prime+y_1}{2},X_{2:k},
\frac{x_1+x_1^\prime+y_{k+1}}{2}, \right. \\
& \qquad \qquad \qquad \qquad \qquad \left. \frac{x_1+x_1^\prime-y_1}{2},X_{2:k}^\prime,
\frac{x_1+x_1^\prime-y_{k+1}}{2} \right) \\
\end{aligned}
\end{equation*}

Use the change of variables $u_1 = \frac{1}{2} (v_{k+1} + v_1)$,
$u_2 = \frac{1}{2} (v_{k+1} - v_1)$.
\begin{equation*}
\begin{aligned}
& i\int C^-_{1,k+1} f^{(k+1)} \left( \frac{X_k+X_k^\prime}{2},V_k\right)
e^{iV_k \cdot (X_k-X_k^\prime)} dV_k = \\
& = i \frac{1}{2^d \pi^{2d}}
\int du_1 du_2 dy_1 dy_{k+1} d\omega 
\mathbf{b} \left( 2|u_2|, \omega\cdot\frac{u_2}{|u_2|}\right)
\times \\
& \qquad \qquad \times
e^{i u_1 \cdot (x_1-x_1^\prime-y_1-y_{k+1})}
e^{- i u_2 \cdot (x_1-x_1^\prime-y_1+y_{k+1})}
  \times \\
& \qquad \qquad \times \gamma^{(k+1)} \left( 
\frac{x_1+x_1^\prime+y_1}{2},X_{2:k},
\frac{x_1+x_1^\prime+y_{k+1}}{2}, \right. \\
& \qquad \qquad \qquad \qquad \qquad \left. \frac{x_1+x_1^\prime-y_1}{2},X_{2:k}^\prime,
\frac{x_1+x_1^\prime-y_{k+1}}{2} \right) \\
\end{aligned}
\end{equation*}
Now let $w = x_1-x_1^\prime-y_1-y_{k+1}$, $z=x_1-x_1^\prime-y_1+y_{k+1}$.
\begin{equation*}
\begin{aligned}
& i\int C^-_{1,k+1} f^{(k+1)} \left( \frac{X_k+X_k^\prime}{2},V_k\right)
e^{i V_k\cdot (X_k-X_k^\prime)} dV_k = \\
& = i \frac{1}{ (2 \pi)^{2d}}
\int du_1 du_2 dw dz d\omega \mathbf{b}
\left(2|u_2|,\omega\cdot\frac{u_2}{|u_2|}\right)
 e^{i u_1 \cdot w}
e^{- i u_2 \cdot z}  \times \\
& \qquad \qquad \times \gamma^{(k+1)} \left( x_1-\frac{w+z}{4},X_{2:k},
\frac{x_1+x_1^\prime}{2}-\frac{w-z}{4}, \right. \\
& \qquad \qquad \qquad \qquad \qquad \left. x_1^\prime + \frac{w+z}{4},X_{2:k}^\prime,
\frac{x_1+x_1^\prime}{2}+\frac{w-z}{4} \right) \\
\end{aligned}
\end{equation*}
Recall that
\begin{equation}
\hat{\mathbf{b}}^\omega (\xi) = \int_{\mathbb{R}^d}
 \mathbf{b} \left( |u|,\omega\cdot\frac{u}{|u|}\right)
e^{-iu \cdot \xi} du
\end{equation}
Then we have
\begin{equation*}
\begin{aligned}
& i\int C^-_{1,k+1} f^{(k+1)} \left( \frac{X_k+X_k^\prime}{2},V_k\right)
e^{i V_k\cdot (X_k-X_k^\prime)} dV_k = \\
& = \frac{i}{ 2^{2d} \pi^{d}}
\int_{\mathbb{S}^{d-1}} d\omega \int_{\mathbb{R}^d} dz 
\hat{\mathbf{b}}^\omega
\left(\frac{z}{2}\right)
 \times \\
& \qquad \times \gamma^{(k+1)} \left( x_1-\frac{z}{4},X_{2:k},
\frac{x_1+x_1^\prime}{2}+\frac{z}{4}, x_1^\prime + \frac{z}{4},X_{2:k}^\prime,
\frac{x_1+x_1^\prime}{2}-\frac{z}{4} \right) \\
\end{aligned}
\end{equation*}

Now the gain term.
\begin{equation*}
\begin{aligned}
& i \int C^+_{1,k+1}f^{(k+1)}
 \left( \frac{X_k+X_k^\prime}{2},V_k\right)
e^{iV_k\cdot (X_k-X_k^\prime)} dV_k = \\ 
& =  i \int
dV_{k+1} d\omega  \mathbf{b}
 \left( |v_{k+1}-v_1|, \omega\cdot\frac{v_{k+1}-v_1}
{|v_{k+1}-v_1|} \right)  e^{i V_k\cdot (X_k-X_k^\prime)}\times\\
&\qquad\qquad\times f^{(k+1)} \left( \frac{x_1+x_1^\prime}{2},
\frac{X_{2:k}+X_{2:k}^\prime}{2},\frac{x_1+x_1^\prime}{2},v_1^*,
V_{2:k},v_{k+1}^* \right) \\
& =  i \frac{1}{(2\pi)^{d(k+1)}} \int
dV_{k+1} dY_{k+1} d\omega  \mathbf{b} \left( |v_{k+1}-v_1|,
\omega\cdot\frac{v_{k+1}-v_1} {|v_{k+1}-v_1|} \right) \times \\
&  \times e^{i V_k \cdot (X_k-X_k^\prime)}
e^{-iV_{2:k}\cdot Y_{2:k}}
e^{-i (v_1+P_\omega (v_{k+1}-v_1)) \cdot y_1}
e^{-i (v_{k+1} - P_\omega (v_{k+1}-v_1)) \cdot y_{k+1}} \times \\
& \times \gamma^{(k+1)} \left( \frac{x_1+x_1^\prime+y_1}{2},
\frac{X_{2:k}+X_{2:k}^\prime+Y_{2:k}}{2},
\frac{x_1+x_1^\prime+y_{k+1}}{2},\right.\\
& \qquad \qquad \left. \frac{x_1+x_1^\prime-y_1}{2},
\frac{X_{2:k}+X_{2:k}^\prime-Y_{2:k}}{2},
\frac{x_1+x_1^\prime-y_{k+1}}{2}\right)  \\
& =  i \frac{1}{(2\pi)^{2d}} \int
dv_1 dv_{k+1} dy_1 dy_{k+1} d\omega  \mathbf{b} \left( |v_{k+1}-v_1|,
\omega\cdot\frac{v_{k+1}-v_1} {|v_{k+1}-v_1|} \right)  \times \\
&  \times e^{i v_1 \cdot (x_1-x_1^\prime)}
e^{-i (v_1+ P_\omega (v_{k+1}-v_1)) \cdot y_1}
e^{-i (v_{k+1} - P_\omega (v_{k+1}-v_1)) \cdot y_{k+1}} \times \\
& \times \gamma^{(k+1)} \left( \frac{x_1+x_1^\prime+y_1}{2}, X_{2:k},
\frac{x_1+x_1^\prime+y_{k+1}}{2},\right.\\
& \qquad \qquad \qquad \qquad \qquad \left. \frac{x_1+x_1^\prime-y_1}{2}, X_{2:k}^\prime,
\frac{x_1+x_1^\prime-y_{k+1}}{2}\right)  \\
\end{aligned}
\end{equation*}

Let $u_1 = \frac{1}{2} (v_{k+1} + v_1)$, $u_2 = \frac{1}{2} (v_{k+1} - v_1)$.
\begin{equation*}
\begin{aligned}
& i \int C^+_{1,k+1} f^{(k+1)} \left( \frac{X_k+X_k^\prime}{2},V_k\right)
e^{iV_k\cdot (X_k-X_k^\prime)} dV_k = \\ 
& =  i \frac{1}{2^d \pi^{2d}} \int
du_1 du_2 dy_1 dy_{k+1} d\omega \mathbf{b} \left( 2|u_2|,
\omega\cdot\frac{u_2}{|u_2|}\right)\times \\
& \qquad \times e^{i u_1 \cdot (x_1-x_1^\prime-y_1-y_{k+1})}
e^{-iu_2 \cdot (x_1-x_1^\prime-R_\omega
(y_1-y_{k+1}))} \times \\
& \qquad \times \gamma^{(k+1)} \left( \frac{x_1+x_1^\prime+y_1}{2}, X_{2:k},
\frac{x_1+x_1^\prime+y_{k+1}}{2},\right.\\
& \qquad \qquad \qquad \qquad \qquad \left. \frac{x_1+x_1^\prime-y_1}{2}, X_{2:k}^\prime,
\frac{x_1+x_1^\prime-y_{k+1}}{2}\right)  \\
\end{aligned}
\end{equation*}

Let $w = x_1-x_1^\prime-y_1-y_{k+1}$, $z=x_1-x_1^\prime-
R_\omega (y_1-y_{k+1})$.
\begin{equation*}
\begin{aligned}
& i \int C^+_{1,k+1} f^{(k+1)} \left( \frac{X_k+X_k^\prime}{2},V_k\right)
e^{iV_k \cdot (X_k-X_k^\prime)} dV_k = \\ 
& =  i \frac{1}{(2 \pi)^{2d}}  \int
du_1 du_2 dw dz d\omega \mathbf{b} \left( 2|u_2|,
\omega\cdot\frac{u_2}{|u_2|}\right) e^{i u_1 \cdot w}
e^{-iu_2 \cdot z} \times \\
& \qquad \times \gamma^{(k+1)} \left( x_1-\frac{1}{2} P_\omega
(x_1-x_1^\prime)-\frac{w+R_\omega (z)}{4}, X_{2:k}, \right. \\
& \qquad \qquad \qquad \qquad \qquad \qquad \left.
\frac{x_1+x_1^\prime}{2}+\frac{1}{2}P_\omega
(x_1-x_1^\prime)-\frac{w-R_\omega (z)}{4},\right.\\
& \qquad \qquad  \left. 
x_1^\prime + \frac{1}{2}P_\omega (x_1-x_1^\prime)+
\frac{w+R_\omega (z)}{4}, X_{2:k}^\prime, \right. \\
& \qquad \qquad \qquad \qquad \qquad \qquad \left.
\frac{x_1+x_1^\prime}{2}-\frac{1}{2} P_\omega 
(x_1-x_1^\prime)+\frac{w-R_\omega (z)}{4}\right)  \\
\end{aligned}
\end{equation*}
We have
\begin{equation*}
\begin{aligned}
& i \int C^+_{1,k+1} f^{(k+1)} \left( \frac{X_k+X_k^\prime}{2},V_k\right)
e^{iV_k \cdot (X_k-X_k^\prime)} dV_k = \\ 
& =  \frac{i}{2^{2d} \pi^{d}}  \int_{\mathbb{S}^{d-1}}
 d\omega \int_{\mathbb{R}^d} dz
 \hat{\mathbf{b}}^\omega \left( \frac{z}{2} \right)  \times \\
& \qquad \times \gamma^{(k+1)} \left( x_1-\frac{1}{2} P_\omega
(x_1-x_1^\prime)-\frac{R_\omega (z)}{4}, X_{2:k}, \right. \\
& \qquad \qquad \qquad \qquad \qquad \qquad \left.
\frac{x_1+x_1^\prime}{2}+\frac{1}{2} P_\omega
(x_1-x_1^\prime)+\frac{R_\omega (z)}{4},\right.\\
& \qquad \qquad  \left. 
x_1^\prime + \frac{1}{2} P_\omega (x_1-x_1^\prime)+
\frac{R_\omega (z)}{4}, X_{2:k}^\prime, \right. \\
& \qquad \qquad \qquad \qquad \qquad \qquad \left.
\frac{x_1+x_1^\prime}{2}-\frac{1}{2} P_\omega
(x_1-x_1^\prime)-\frac{R_\omega (z)}{4}\right)  \\
\end{aligned}
\end{equation*}
\end{proof}

We now turn to the Boltzmann equation. Here it suffices to notice that
a solution $f$ of the Boltzmann equation is just a \emph{factorized}
solution of the Boltzmann hierarchy, i.e.
$f^{(k)} = f^{\otimes k}$. 
Using Proposition \ref{prop:App-BH-TR} and Proposition 
\ref{prop:App-BH-Coll}, we obtain:
\begin{corollary}
\label{cor:App-BE-TR}
Let $f \in L^1 \left([0,T],L^2 (\mathbb{R}^d \times \mathbb{R}^d)
\right)$ and let $\gamma$ denote the inverse Wigner transform
of $f$. Then if
\begin{equation}
\left( \partial_t + v \cdot \nabla_x \right) f = g
\end{equation}
holds in the sense of distributions, then we have
\begin{equation}
\left( i \partial_t + \frac{1}{2} \left( \Delta_x -
\Delta_{x^\prime}\right)\right) \gamma = i \int_{\mathbb{R}^d}
g \left( t,\frac{x+x^\prime}{2},v\right)
e^{iv\cdot (x-x^\prime)} dv
\end{equation}
in the sense of distributions.
\end{corollary}
\begin{corollary}
\label{cor:App-BE-Coll}
Let $f(x,v)$ be a Schwartz function, and let $\gamma$ denote its
inverse Wigner transform. Then
\begin{equation}
\begin{aligned}
& i \int Q(f,f) \left( \frac{x+x^\prime}{2},v\right)
e^{i v \cdot (x-x^\prime)} dv = \\
& = \frac{i}{2^{2d} \pi^d} \int_{\mathbb{S}^{d-1}} d\omega
\int_{\mathbb{R}^d} dz \hat{\mathbf{b}}^\omega
\left( \frac{z}{2}\right) \times \\
&\qquad \times \left\{
\gamma\left( x - \frac{1}{2} P_\omega (x-x^\prime) -
\frac{R_\omega (z)}{4},
x^\prime + \frac{1}{2} P_\omega (x-x^\prime) +
\frac{R_\omega (z)}{4}\right) \times \right. \\
& \qquad \times \gamma \left( \frac{x+x^\prime}{2}+
\frac{1}{2} P_\omega (x-x^\prime)+\frac{R_\omega (z)}{4},
\frac{x+x^\prime}{2}-\frac{1}{2} P_\omega (x-x^\prime)
-\frac{R_\omega (z)}{4}\right) \\
&\qquad \left.-\gamma \left( x-\frac{z}{4},x^\prime + \frac{z}{4}\right)
\gamma \left( \frac{x+x^\prime}{2}+\frac{z}{4},
\frac{x+x^\prime}{2}-\frac{z}{4}\right) \right\}
\end{aligned}
\end{equation}
\end{corollary}

\end{appendix}

\bibliography{wigner-paper}

\begin{bibdiv}
\begin{biblist}

\bib{AMUXY2011}{article}{
      author={Alexandre, R.},
      author={Morimoto, Y.},
      author={Ukai, S.},
      author={Xu, C.-J.},
      author={Yang, T.},
       title={Global existence and full regularity of the \uppercase{B}oltzmann
  equation without angular cutoff},
        date={2011},
     journal={Comm. Math. Phys.},
      volume={304},
      number={2},
       pages={513\ndash 581},
}

\bib{AMUXY-Loc}{article}{
      author={Alexandre, R.},
      author={Morimoto, Y.},
      author={Ukai, S.},
      author={Xu, C.-J.},
      author={Yang, T.},
       title={Local existence with mild regularity for the
  \uppercase{B}oltzmann equation},
        date={2013},
     journal={Kinetic and Related Models},
      volume={6},
      number={4},
       pages={1011\ndash 1041},
}

\bib{Ar2011}{article}{
      author={Arsenio, D.},
       title={On the global existence of mild solutions to the
  \uppercase{B}oltzmann equation for small data in \uppercase{$L^D$}},
        date={2011},
     journal={Comm. Math. Phys.},
      volume={302},
      number={2},
       pages={453\ndash 476},
}

\bib{CIP1994}{book}{
      author={Cercignani, C.},
      author={Illner, R.},
      author={Pulvirenti, M.},
       title={The mathematical theory of dilute gases},
   publisher={Springer Verlag},
        date={1994},
}

\bib{PC2010}{article}{
      author={Chen, T.},
      author={Pavlovi{\' c}, N.},
       title={On the \uppercase{C}auchy problem for focusing and defocusing
  \uppercase{G}ross-\uppercase{P}itaevskii hierarchies},
        date={2010},
     journal={Discr. Contin. Dyn. Syst. A},
      volume={27},
      number={2},
       pages={715\ndash 739},
}

\bib{DPL1989}{article}{
      author={DiPerna, R.~J.},
      author={Lions, P.-L.},
       title={On the \uppercase{C}auchy problem for \uppercase{B}oltzmann
  equations: \uppercase{G}lobal existence and weak stability},
        date={1989},
     journal={Ann. Math.},
      volume={130},
      number={2},
       pages={321\ndash 366},
}

\bib{Du2008}{article}{
      author={Duan, Renjun},
       title={On the \uppercase{C}auchy problem for the \uppercase{B}oltzmann
  equation in the whole space: \uppercase{G}lobal existence and uniform
  stability in \uppercase{$L^2_\xi (H^N_x)$}},
        date={2008},
     journal={Journal of Differential Equations},
      volume={244},
      number={12},
       pages={3204\ndash 3234},
}

\bib{ESY2006}{article}{
      author={Erd{\" o}s, L.},
      author={Schlein, B.},
      author={Yau, H.-T.},
       title={Derivation of the \uppercase{G}ross-\uppercase{P}itaevskii
  hierarchy for the dynamics of \uppercase{B}ose-\uppercase{E}instin
  condensate},
        date={2006},
     journal={Comm. Pure Appl. Math.},
      volume={59},
      number={12},
       pages={1659\ndash 1741},
}

\bib{ESY2007}{article}{
      author={Erd{\" o}s, L.},
      author={Schlein, B.},
      author={Yau, H.-T.},
       title={Derivation of the cubic non-linear \uppercase{S}chr{\" o}dinger
  equation from quantum dynamics of many-body systems},
        date={2007},
     journal={Invent. math.},
      volume={167},
       pages={515\ndash 614},
}

\bib{ES2001}{article}{
      author={Erd{\" o}s, L.},
      author={Yau, H.-T.},
       title={Derivation of the nonlinear \uppercase{S}chr{\" o}dinger equation
  form a many body \uppercase{C}oulomb system},
        date={2001},
     journal={Adv. Theor. Math. Phys.},
      volume={5},
       pages={1169\ndash 1205},
}

\bib{GSRT2014}{article}{
      author={Gallagher, I.},
      author={Saint-Raymond, L.},
      author={Texier, B.},
       title={From \uppercase{N}ewton to \uppercase{B}oltzmann: Hard spheres
  and short-range potentials},
        date={2014},
     journal={Zurich Lec. Adv. Math.},
}

\bib{GS2011}{article}{
      author={Gressman, P.~T.},
      author={Strain, R.~M.},
       title={Global classical solutions of the \uppercase{B}oltzmann equation
  without angular cut-off},
        date={2011},
     journal={J. Amer. Math. Soc.},
      volume={24},
      number={3},
       pages={771\ndash 847},
}

\bib{Gu2003}{article}{
      author={Guo, Y.},
       title={Classical solutions to the \uppercase{B}oltzmann equation for
  molecules with an angular cutoff},
        date={2003},
     journal={Archive for Rational Mechanics and Analysis},
      volume={169},
      number={4},
       pages={305\ndash 353},
}

\bib{HS1955}{article}{
      author={Hewitt, E.},
      author={Savage, L.~J.},
       title={Symmetric measures on \uppercase{C}artesian products},
        date={1955},
     journal={Trans. Amer. Math. Soc.},
      volume={80},
       pages={470\ndash 501},
}

\bib{KS1984}{article}{
      author={Kaniel, S.},
      author={Shinbrot, M.},
       title={The \uppercase{B}oltzmann equation: \uppercase{I}.
  \uppercase{U}niqueness and local existence},
        date={1978},
     journal={Communications in Mathematical Physics},
      volume={58},
      number={1},
       pages={65\ndash 84},
}

\bib{Ki1975}{thesis}{
      author={King, F.},
       title={\uppercase{BBGKY} hierarchy for positive potentials},
        type={Ph.D. Thesis},
        date={1975},
}

\bib{KM2008}{article}{
      author={Klainerman, S.},
      author={Machedon, M.},
       title={On the uniqueness of solutions to the
  \uppercase{G}ross-\uppercase{P}itaevskii hierarchy},
        date={2008},
     journal={Comm. Math. Phys.},
      volume={279},
      number={1},
       pages={169\ndash 185},
}

\bib{L1975}{incollection}{
      author={Lanford, O.~E.},
       title={Time evolution of large classical systems},
        date={1975},
   booktitle={Dynamical systems, theory and applications},
      editor={Moser, J.},
      series={Lecture Notes in Physics},
      volume={38},
   publisher={Springer Berlin Heidelberg},
       pages={1\ndash 111},
}

\bib{Uk1974}{article}{
      author={Ukai, S.},
       title={On the existence of global solutions of mixed problem for the
  non-linear \uppercase{B}oltzmann equation},
        date={1974},
     journal={Proc. Japan Acad.},
      volume={50},
      number={3},
       pages={179\ndash 184},
}

\bib{UY2006}{article}{
      author={Ukai, S.},
      author={Yang, T.},
       title={The \uppercase{B}oltzmann equation in the space \uppercase{$L^2
  \cap L^\infty_\beta$}: \uppercase{G}lobal and time-periodic solutions},
     journal={Analysis and Applications},
      volume={4},
      number={3},
       pages={263\ndash 310},
}

\end{biblist}
\end{bibdiv}

\end{document}